\documentclass{article}
\usepackage[a4paper, total={6in, 8in}]{geometry}

\usepackage{graphicx}%
\usepackage{multirow}%
\usepackage{amsmath,amssymb,amsfonts}%
\usepackage{amsthm}%
\usepackage{mathrsfs}%
\usepackage[title]{appendix}%

\usepackage{bm}
\RequirePackage[colorlinks,citecolor=blue,urlcolor=blue]{hyperref}
\usepackage[numbers]{natbib}
\usepackage[dvipsnames]{xcolor}
\usepackage{xcolor, soul} 
\usepackage[normalem]{ulem}

\newtheorem{theorem}{Theorem}[section]%
\newtheorem{proposition}{Proposition}[section]%
\newtheorem{lemma}{Lemma}[section]%
\newtheorem{corollary}{Corollary}[section]

\newtheorem{remark}{Remark}%

\title{Asymptotic fluctuations of smooth linear statistics of independently perturbed lattices.}

\author{Gabriel Mastrilli $^{1, 2}$}
\date{
	$^1$ Univ Rennes, Ensai, CNRS, CREST \\
	$^2$ Inria, Paris, France.
}

\begin{document}
	\maketitle

	\begin{abstract} We consider the hyperuniform model of {\it d}-dimensional integer lattice perturbed by independent random variables and we investigate the large scale asymptotic fluctuations of smoothed versions of the usual counting statistics, specifically of linear statistics associated to a smooth function with rapid decay at infinity. We highlight three distinct classes of limit, depending on the dimension {\it d} and on the tails of the perturbations. On the one hand, we establish that for dimensions larger than two, central limit theorems hold under mild assumptions on the perturbations. This confirms numerical observations from physics, suggesting that even for highly correlated hyperuniform models, large dimensions favor asymptotic normality. On the other hand, in dimension one, the limiting distribution can be Gaussian, non-Gaussian with finite moments of all orders, or stable with parameter strictly between one and two. These two latter results represent rare examples of non-Gaussian limits for smooth linear statistics of hyperuniform point processes of Classes I and~II.
		~
	\newline \newline
	
	
	\end{abstract}
	
	\section{Introduction}
	
	\subsection{Motivation and related works}\label{sec:11}
	
	Hyperuniform point processes~\cite{torquato2003local}, which exhibit more regularity than the Poisson point process, have attracted considerable attention in recent years~\cite{torquato2018hyperuniform}. A central class of such processes consists in perturbed lattices $\{x+ Z(x) \mid x \in \mathbb{Z}^d\}$, where $Z$ is a random field~\cite{gabrielli2003generation, gabrielli2004point, dereudre2024non}. The flexibility of this construction~\cite{gabrielli2008tilings, torquato2003local, kim2018effect, klatt2020cloaking} makes it possible to replicate the behavior of more intricate models and aids in their theoretical understanding~\cite{sodin2006random, holroyd2013insertion, peres2014rigidity, arai2019rigidity, yakir2021fluctuations, yakir2022recovering, lachieze2024hyperuniformity, butez2024wasserstein, dereudre2024non}. The present paper explores how hyperuniformity, which implies potentially strong second-order correlations between points, affects higher-order correlations, through the study of the fluctuations of linear statistics of independently perturbed lattices in the large-scale limit. It is directly motivated by the numerical observations of~\cite{torquato2021local}, who observed that hyperuniform point processes in high ambient dimensions $d$ still exhibit large-scale Gaussian fluctuations. However, in dimension one, such fluctuations may fail to appear for strongly hyperuniform point processes. We provide mathematical evidence for these observations and refine the dimension-one case by proving a transition from Gaussian to stable fluctuations, with an unexpected and non-standard behavior at the transition.

	Specifically, we studies the asymptotic behavior, as $r \to \infty$, of the following linear statistic associated with a function $f$:
	\begin{equation*}
		T_r(f) := \sum_{x \in \Phi} f(x/r),
	\end{equation*}
	where $\Phi :=  \{x+ U + \xi_x \mid x \in \mathbb{Z}^d\}$ is a stationary independently perturbed lattice. Here, the i.i.d. random variables $(\xi_x)_{x \in \mathbb{Z}^d}$ are called perturbations, and $U$ is uniformly distributed over $[-1/2, 1/2]^d$, independent of $(\xi_x)_{x \in \mathbb{Z}^d}$.
	
	It is well known that the asymptotic behavior of $T_r(f)$ as $r \to \infty$ provides insights into the large-scale properties of the point process $\Phi$. For instance, if $f = \mathbf{1}_{B(0, 1)}$, then the variance of $T_r(f)$ corresponds to the variance of the number of points of $\Phi$ inside the Euclidean ball $B(0, r)$. In our setting, this variance scales, as $r \to \infty$, smaller than the volume of $B(0, r)$, which is the hyperuniformity property~\cite{torquato2018hyperuniform, ghosh2017fluctuations, coste2021order}. For general $f$, $T_r(f)$ can be viewed as a natural extension of points counting, and one can study the asymptotics of all moments of $T_r(f)$ or its limiting distribution. Another motivation for studying the fluctuations of $T_r(f)$ arises from recent advances in spectral inference of spatial point processes, which rely on central limit theorems for linear statistics~\cite{hawat2023estimating, klatt2022genuine, rajala2023fourier, yang2023fourier, grainger2023spectral, mastrilli2024estimating}.
	
	Numerical observations in~\cite{torquato2021local} suggest that such central limit theorems are facilitated by a large ambient dimension $d$, but in dimension $1$, they may fail for highly hyperuniform point processes of Class I~\cite{torquato2018hyperuniform}. In Theorem~\ref{thm_synthese}, we confirm both observations for $T_r(f)$ when $f$ is a smooth function. In dimensions $d \geq 3$, $T_r(f)$ always exhibits asymptotic Gaussian fluctuations, provided that the perturbations are non-degenerate. In dimension $2$, the same holds under mild assumptions. However, the picture is different in dimension $1$, where Class I perturbed lattices lead to stable fluctuations, refining the observations of~\cite{torquato2021local}. Furthermore, we exhibit unexpected fluctuations, that are not Gaussian nor stable, for moderately hyperuniform perturbed lattices of Class II~\cite{torquato2018hyperuniform}.
	
	To the authors' knowledge, examples of non-Gaussian large-scale limits for smooth linear statistics of infinite point processes have not been theoretically highlighted before. Previous works have established stable fluctuations of stable-weighted linear statistics for determinantal point processes~\cite{aburayama2023scaling} and Poisson-shot noise with non-integrable response functions~\cite{baccelli2015scaling}. However, in these cases, the non-normality arises due to the stable nature of the weights in the first case and to the non-integrability of the response function in the second, rather than from intrinsic properties of the point process itself, as is the case for Class I perturbed lattices in dimension $1$.
	
	\smallskip
	Before detailing our results in Sections~\ref{sec:summary} and~\ref{sec:proof_idea}, we briefly put our proof techniques in perspective with respect to those in the existing literature.
	
	To study the large-scale fluctuations of statistics such as $T_r(f)$, two main settings emerge, depending on whether the variance of $T_r(f)$ is asymptotically bounded or not. When the variance diverges, general central limit theorems can be proved using standard cumulant methods~\cite{soshnikov2002gaussian, heinrich2016strong, blaszczyszyn2019limit, doring2022method}. When the variance is bounded, proofs usually rely on finer considerations of the cumulants' structure of the point process~\cite{costin1995gaussian, sodin2004random, soshnikov2002gaussian, rider2006complex, gaultier2016fluctuations, haimi2024normality, krishnapur2024stationary}. In our case, depending on the dimension $d$ and the tail behavior of the perturbations, we encounter both previous settings, along with a third peculiar situation where the variance converges to zero, leading to specific techniques.
	
	To derive the cumulants of $T_r(f)$ for general perturbations, we express them in the Fourier domain, leveraging the smoothness of $f$. This representation allows us to obtain a general central limit theorem in dimension $d \geq 3$, a situation where the variance of $T_r(f)$ always diverges as $r \to \infty$, facilitating the proof. In dimension $d = 2$, the variance may be asymptotically bounded. For such setting, we establish the asymptotic normality of $T_r(f)$ by analyzing the combinatorial structure of the cumulants, a key argument encountered in the literature in proving central limit theorems for processes such as the Ginibre or the zero set of the GAF~\cite{costin1995gaussian, sodin2004random, soshnikov2002gaussian, rider2006complex, gaultier2016fluctuations, haimi2024normality, krishnapur2024stationary}. In dimension $1$, $T_r(f)$ has a bounded variance when the perturbations belong to the class of Cauchy-like heavy-tailed distributions. To handle this challenging situation, our general expression of the cumulants in the Fourier domain becomes useful. This case corresponds to a Class II perturbed lattice, for which we prove that $T_r(f)$ exhibits non-Gaussian and non-stable fluctuations. The underlying reason is that the absence of moments for the perturbations leads to an irregular characteristic function at zero, thereby disrupting the combinatorial arguments used in dimension $2$. Finally, still in dimension $1$, for lattices perturbed by $\alpha$-stable random variables with $\alpha \in (1,2)$, the variance of $T_r(f)$ even converges to $0$ and the cumulant method fails. Instead, we analyze the characteristic function of $T_r(f)$ directly, showing that in this setting $T_r(f)$ can be approximated by the local shot noise $\sum_{x \in \Phi \cap B(0, r)} f'(x/r)\xi_x/r$.
	
	The remainder of this introduction states our results in Section~\ref{sec:summary}, outlines the proofs strategies in Section~\ref{sec:proof_idea}, and introduces notations in Section~\ref{sec_notations}. The rest of the article is organized as follows. Before extending all results to stationary perturbed lattices in Section~\ref{sec_stat}, we first focus for technical reasons in Sections~\ref{sec_exp_var},~\ref{sec_gauss}, and~\ref{sec_non_gauss}  on non-stationary perturbed lattices. Section~\ref{sec_exp_var} derives an asymptotic expansion of the cumulants of their linear statistics as $r \to \infty$. Then, Section~\ref{sec_gauss} proves central limit theorems for $d \geq 2$. Sections~\ref{sec_alpha_1} and~\ref{sec_no_clt} analyze non-stable and $\alpha$-stable limits in dimension $1$. Finally, Section~\ref{sec_technical} contains key lemmas related to the Poisson summation formula. For completeness, Appendix~\ref{sec_hyp_phi} provides proofs of standard facts regarding the link between the behavior at zero of the characteristic function of a random variable and its tail and Appendix~\ref{app_slwy} contains results about slowly varying functions.
	
	\subsection{Summary of the results}\label{sec:summary}
	
	In this section we summarize our results in Theorem~\ref{thm_synthese} and then discuss their assumptions. We consider the following linear statistics, where $f$ is a Schwartz function (see Section~\ref{sec_notations}), 
	\begin{equation}\label{eq:Tr}
		T_r^i := \sum_{x \in \Phi^i} f(x/r),~i \in \{0, 1\},
	\end{equation}
	and where the non stationary and stationary perturbed lattices are
	\begin{equation}\label{eq:def_phi}
		\Phi^{0} :=  \{x+ \xi_x|~x \in \mathbb{Z}^d\},~		\Phi^{1} :=  \{x+ U + \xi_x|~x \in \mathbb{Z}^d\}.
	\end{equation}
	The random variables $(\xi_x)_{x \in \mathbb{Z}^d}$ are i.i.d. and take values in $\mathbb{R}^d$. Moreover, $U$ is uniformly distributed over $[-1/2, 1/2]^d$ and is independent of $(\xi_x){x \in \mathbb{Z}^d}$. Note that $\Phi^1$ is the perturbed lattice $\Phi$ considered in Section~\ref{sec:11}. We introduce the non-stationary perturbed lattice $\Phi^0$ as it is more convenient for the proofs, before generalizing to the stationary setting of~$\Phi^1$.
	
	The next theorem summarizes the main results of the paper, namely Theorems \ref{thm_tcl_var_unbouded}, \ref{thm_tcl_var_bouded}, \ref{thm_tcl_var_unbounded_2_ss}, \ref{thm_tcl_var_unbounded_2}, \ref{thm_no_clt_alpha_1}, \ref{thm_cv_alpha_stable} and~\ref{thm_cv_stat}.
	
	\medskip \begin{theorem}\label{thm_synthese}
		Let $d \geq 1$ and $f$ be a Schwartz function. We consider i.i.d. random variables~$(\xi_x)_{x \in \mathbb{Z}^d}$ with characteristic function $\varphi$. We assume that $f$ is non-null and that the perturbations are not a.s. constant and we consider for $i \in \{0, 1\}$ the linear statistics $T_r^{i}$ defined in equation~\eqref{eq:Tr}.
		\begin{enumerate}
			\setlength\itemsep{0.5em}
			\item \label{i_d3_thm_syn} If $d \geq 3$, 
			$$\operatorname{Var}[T_r^i]^{-1/2} \left(T_r^i - r^d \int_{\mathbb{R}^d} f(x) dx\right) \xrightarrow[r \to \infty]{d} \mathcal{N}(0, 1).$$
				\setlength\itemsep{0.5em}
			\item \label{i_2_bound_thm_syn} If $d =2 $ and $\mathbb{E}[|\xi_{0}|^2] < \infty$, then, denoting by $\xi_0'$ an independent copy of $\xi_0$,
			$$T_r^i - r^2 \int_{\mathbb{R}^2} f(x) dx \xrightarrow[r \to \infty]{d} \mathcal{N}\left(0, \frac12\int_{\mathbb{R}^2} |\mathcal{F}[f](x)|^2 \mathbb{E}[|(\xi_0 - \xi_0').x|^2] dx\right).$$
			\item\label{i_2_unbound_thm_syn} If $d =2 $ and $\mathbb{E}[|\xi_0|^\nu] = \infty$ for some $0 < \nu < 2$, then there exists a diverging positive sequence $(r_n)_{n \geq 1}$ such that
			$$\operatorname{Var}[T_{r_n}^i]^{-1/2} \left(T_{r_n}^i - r_n^2 \int_{\mathbb{R}^2} f(x) dx\right) \xrightarrow[r \to \infty]{d} \mathcal{N}(0, 1).$$
		\end{enumerate}
		For the next results, we suppose that $1 - \varphi(x) \sim L(|x|)|x|^{\alpha}$ as $|x| \to 0$, where $\alpha \in (0, 2]$ and $L$ is slowly varying with a limit $c \in [0, \infty]$ at zero (see Section~\ref{sec_notations}).
		\begin{enumerate}
			\setlength\itemsep{0.5em}  
			\setcounter{enumi}{3}
			\item \label{i_2_phi_thm_syn}  If $d \geq 2$ with no further assumption or $d = 1$ with $\alpha < 1$ or $\alpha = 1$ and $c = \infty$, then
				$$L(1/r)^{-1/2} r^{\frac{\alpha-d}{2}} \left(T_r^i - r^d \int_{\mathbb{R}^d} f(x) dx\right) \xrightarrow[r \to \infty]{d} \mathcal{N}\left(0, \int_{\mathbb{R}^d} |\mathcal{F}[f](x)|^{2} |x|^{\alpha} dx\right).$$
			\item\label{i_11_thm_syn} If $d = \alpha = 1$ and $c \in (0, \infty)$, i.e.  $1 - \varphi(x) \sim c|x|$ as $|x| \to 0$, then $T_r^i - r \int_{\mathbb{R}} f(x) dx$ converges in distribution toward a random variable, which is not Gaussian in general, but has moments of all orders.
			\item \label{i_1_satble_thm_syn} If $d = 1$, $1 < \alpha \leq 2$ and $c \in (0, \infty)$, i.e. $1 - \varphi(x) \sim c|x|^{\alpha}$ as $|x| \to 0$, then
			 $$r^{\frac{\alpha-1}{\alpha}} \left(T_r^i - r \int_{\mathbb{R}} f(x) dx\right) \xrightarrow[r \to \infty]{d} \left(c\int_{\mathbb{R}} |f'(x)|^{\alpha} dx\right)^{1/\alpha} S_{\alpha},$$
			 where $S_{\alpha}$ is $\alpha$-stable with characteristic function $\mathbb{E}[e^{2\bm{i}\pi k S_{\alpha}}] = e^{-|k|^{\alpha}}$.
		\end{enumerate}
	\end{theorem}

	Before detailing the points of the previous theorem, it is enlightening to discuss the assumption regarding the behavior of the characteristic function $\varphi$ of the perturbations near zero. First of all, this behavior is related to the tail of the perturbations. Indeed, a large $\alpha$ corresponds to a relatively fast tail decay, while a small $\alpha$ implies a slow decay (see Lemma~\ref{lemma_char_and_moment}). Moreover, the parameter $\alpha$ is connected to the asymptotic behavior of the variance of $T_r^i$, which scales as $L(1/r)r^{d-\alpha}$ (see Proposition~\ref{prop_var}), and can thus be interpreted as the strength of correlation between the points of the perturbed lattice. This motivates the order of the items in the previous theorem.
	
	When $d \geq~3$, the fluctuations of $T_r^i$ are always asymptotically Gaussian, with no assumption on the perturbations $(\xi_x)_{x \in \mathbb{Z}^d}$, except their non degeneracy (see the point~\ref{i_d3_thm_syn}). An explanation for this striking phenomenon is that, due to the smoothness of the function $f$, the cumulants of $T_r^i$ behave like those of Brillinger-mixing point processes~\cite{bismut1982processus, heinrich2012asymptotic} (see Proposition~\ref{prop_kappa}).
	
	However, when $d =2$, it is less clear if such a general result holds. Nevertheless, for square integrable perturbations, although the variance of $T_r^i$ is bounded, a central limit theorem holds (see the point~\ref{i_2_bound_thm_syn}). Moreover, when the perturbations are heavy tailed, there is at least convergence along a sub-sequence (see the point~\ref{i_2_unbound_thm_syn}). Since the behavior of the moments of $T_r^0$ strongly depend in dimension $2$ on the behavior near zero of the characteristic function $\varphi$ of the perturbations, it seems challenging to obtain the convergence of the full sequence without any further assumptions. But, with a mild assumption on the behavior near zero of $\varphi$, asymptotic normality is guaranteed (see the point~\ref{i_2_phi_thm_syn}). Note that in this former case, the variance is, in particular, allowed to diverges slowly, leading to possibly logarithmic rate of convergence.
	
	When $d = 1$, the fluctuations of $T_r^i$ can be non Gaussian. In particular, for $\alpha$-stable perturbations with parameter $\alpha \in (1, 2)$, in dimension $d = 1$, the limiting distribution is $\alpha$-stable  (see point~\ref{i_1_satble_thm_syn}). It is noteworthy that the stable nature of the fluctuations emerges only asymptotically, since the moments of $T_r^i$ are all finite. Moreover, the rate $r^{(\alpha - 1)/\alpha}$ is non-standard when $\alpha < 2$. Nevertheless, when $\alpha = 2$, we recover a Gaussian limit with the convergence rate given by the standard deviation.
	 
	Finally, a transition appears at $\alpha = d = 1$, where the convergence is toward a non-standard random variable, that is not Gaussian in general, but also not stable with parameter strictly less than two, as it possesses moments of all orders (see the point~\ref{i_11_thm_syn}).
	
	\medskip \begin{remark}\label{rmk_smooth_f}
		To handle general perturbations, we impose smoothness and decay assumptions on the function $f$, through the Schwartz assumption. Indeed, without them, interplay between the irregularity of the points and $f$ may lead to case by case results concerning, at least, the second order properties of the linear statistics of perturbed lattices associated to $f$. Specifically, it has been proved in~\cite{yakir2021fluctuations} that, even in the case of Gaussian perturbations $(\xi_{x})_{x \in \mathbb{Z}^d}$ in dimension $d \geq 3$, one can construct an $L^1(\mathbb{R}^d)$ function with $L^2(\mathbb{R}^d)$ gradient such that the variance of $T_r^1$, does not scale as $r^{d-2}$, contrary to what happens with smooth functions (see also~\cite{dereudre2024non} studying the variance of $T_r^1$ where $f = \mathbf{1}_{B(0, 1)}$ for dependently perturbed lattices and~\cite{nazarov2011fluctuations} in the context of zeroes of GAF functions).
	\end{remark}

	\subsection{Main ideas of the proofs}\label{sec:proof_idea}
	
	The first sections, namely Sections~\ref{sec_exp_var},~\ref{sec_gauss} and~\ref{sec_non_gauss} deal with the non-stationary perturbed lattice $\Phi^0$, defined in equation~\eqref{eq:def_phi}, in order to exploit the independence between the perturbations, without the common translation by the uniform random variable $U$, involved in the stationary perturbed lattice~$\Phi^1$.
	
	In Section~\ref{sec_exp_var}, using the Poisson summation formula, we exhibit in Proposition~\ref{prop_kappa} the main term of the asymptotic expansion of the cumulants of $T_r^0$ as $r \to \infty$. To do so, we extend the Fourier-method of~\cite{yakir2021fluctuations} to general i.i.d. perturbations $(\xi_x)_{x \in \mathbb{Z}^d}$ and to cumulants of all orders, leveraging on the Schwartz assumption on $f$ (see also the recent paper~\cite{dereudre2024non} which consider variance of linear statistics associated to $f = \mathbf{1}_{B(0, 1)}$ for possibly dependently perturbed lattices). 
	
	In Section~\ref{sec_clt_inf_d}, we exploit the expression of the cumulants derived previously to upper bound the high order cumulants of $T_r^0$ and lower bound its variance by a large enough quantity when $d \geq 3$. Gathering both results, we obtain a central limit theorem when $d \geq 3$ (refer to Theorem~~\ref{thm_tcl_var_unbouded}). 
	
	The case of bounded variance in dimension two is then considered in Section~\ref{sec_alpha_2}. To handle it, we use a second order Taylor expansion of the characteristic function of the perturbation inside the cumulants of $T_r^0$ to be able to compute their limits and show that algebraic cancellations occur for the cumulants of order larger than three (see Theorem~\ref{thm_tcl_var_bouded}). Then, we investigate the case where the perturbations possess an infinite moment of order strictly smaller than $2$. By lower bounding the variance of $T_r^0$ with a quantity that depends on the tails of the perturbations, that decay slowly because of their lack of moments, we prove a central limit theorem for a sub-sequence (refer to Theorem~\ref{thm_tcl_var_unbounded_2_ss}). Finally, by assuming that the characteristic function of the perturbations admits a specific behavior near zero, we are able to compute an asymptotic equivalent (and not only a lower bound) of the variance of $T_r^0$ and prove that its fluctuations are asymptotically Gaussian, even for perturbations with an infinite moment of order 2, but with all moments of order strictly smaller than 2 (refer to Theorem~\ref{thm_tcl_var_unbounded_2}).
	
	Then, we exhibit a first example of non-Gaussian limit in Section~\ref{sec_alpha_1} by considering the dimension 1 and perturbations with a characteristic function $\varphi$ satisfying $1 - \varphi(x) \sim |x|$ as $|x| \to 0$. In this setting, the variance of $T_r^0$ is bounded but, contrary to the case of bounded variance in dimension 2, we prove that no general asymptotic cancellations occur in the cumulants. However, they converge toward consistent cumulants, and, as a consequence, we obtain in Theorem~\ref{thm_no_clt_alpha_1} convergence toward a non-Gaussian random variable, characterized by its moments.
	
	In Section~\ref{sec_no_clt}, still in dimension 1, we consider the case where the variance of $T_r^0$ converges to zero at a rate of $r^{1 - \alpha}$, where $\alpha \in (1, 2]$, assuming that $1 - \varphi(x) \sim |x|^{\alpha}$ as $|x| \to 0$. Since the fluctuations are $\alpha$-stable with possibly $\alpha < 2$, we cannot directly use a cumulant-based proof. Instead, we analyze the limit of the characteristic function of $T_r^0$. The argument consists of two main steps. The first step is to quantify that the fluctuations of $T_r^0$ are dominated by those of the local linear statistic $\sum_{x \in \mathbb{Z}, |x| \leq r^{\gamma}} f((x+\xi_x)/r)$ for a well-chosen parameter $\gamma>1$. The second step is to show, via a Taylor expansion, that this local linear statistic fluctuates like the shot noise $\sum_{|x| \leq r^{\gamma}} f'(x/r) \xi_x/r$, which exhibits $\alpha$-stable fluctuations.
	
		\subsection{General notations}\label{sec_notations}
	
	We consider the Euclidean space $\mathbb{R}^d$ of dimension $d \geq 1$. For $a, b \in \mathbb{C}^d$, $a \cdot b = \sum_{i = 1}^d a_i \overline{b_i}$ denotes the Hermitian scalar product of $\mathbb{C}^d$ between $a$ and $b$. The associated Euclidean norm is denoted by $|a|$ and for a complex number $z \in \mathbb{C}$, $\Re z$ is its real part. We denote $\bm{i} = \sqrt{-1} \in \mathbb{C}$. We use the notation $\sum_{x_1, \dots, x_n \in \mathbb{Z}^d}^{\neq}$ to indicate that the sum runs over the distinct points $x_1, \dots, x_n$: $\forall i \neq j,~x_i \neq x_j$.
	
	For $1\leq p < \infty$, we denote by $L^{p}(\mathbb{R}^d)$ the space of measurable functions $f: \mathbb{R}^d \to \mathbb{R}$ such that $\int_{\mathbb{R}^d} |f(x)|^p dx < \infty$. The $L^p$-norm of $f \in L^p(\mathbb{R}^d)$ is defined by $\|f\|_p^p:= \int_{\mathbb{R}^d} |f(x)|^p dx$. We denote by $L^{\infty}(\mathbb{R}^d)$ the space of functions $f: \mathbb{R}^d \to \mathbb{R}$ such that $\|f\|_{\infty}:= \sup_{x \in \mathbb{R}^d} |f(x)| < \infty$. When it is well defined, the convolution between two functions $f_1$ and $f_2$ is denoted by $f_1 \ast f_2 = \int_{\mathbb{R}^d} f_1(y) f_2(\cdot - y) dy$. More generally, when it is well defined, the convolution between $n \geq 2$ functions $f_1, \dots, f_n$ is defined recursively by $\circledast_{i = 1}^n f_i = f_n \ast \{\circledast_{i = 1}^{n-1} f_i\}$ with the convention $\circledast_{i = 1}^{1} f_i = f_1$. Recall that the Young's inequality for the convolution states that $\|f \ast  g\|_r \leq \|f\|_p \|g\|_q$ for $(p,q, r) \in [1, \infty]$ such that $1+1/r = 1/p+1/q$. For $p = 2$, the scalar product between the $L^{2}(\mathbb{R}^d)$ functions $f_1$ and $f_2$ is $\langle f_1, f_2 \rangle = \int_{\mathbb{R}^d} f_1(x) \overline{f_2(x)} dx$. In view of the Poisson summation formula, we adopt the following definition for the Fourier transform of a function $f \in L^{1}(\mathbb{R}^d)$: 
	$$\forall k \in \mathbb{R}^d,\quad \mathcal{F}[f](k):= \int_{\mathbb{R}^d} f(x) e^{2\pi \bm{i} k \cdot x} dx.$$ As usual, the Fourier transform is extended to $L^2(\mathbb{R}^d)$ functions thanks to the Plancherel theorem~\cite{folland2009fourier}: $\forall f_1, f_2 \in L^2(\mathbb{R}^d)$, $\langle f_1, f_2 \rangle = \langle \mathcal{F}[f_1], \mathcal{F}[f_2] \rangle$. For coherence with the convention of the Fourier transform on $L^{2}(\mathbb{R}^d)$ defined above, we adopt the following definition of the characteristic function $\varphi$ of a random variable $X$ with value in $\mathbb{R}^d$:
	$$\forall t \in \mathbb{R}^d, \varphi(t) := \mathbb{E}[e^{2\pi \bm{i} t\cdot X}].$$
	
	A function $f: \mathbb{R}^d \to \mathbb{R}$ is in the Schwartz space $\mathcal{S}(\mathbb{R}^d)$ if $f$ is infinitely differentiable and if for all multi-indexes $(\beta_1, \beta_2) \in (\mathbb{N}^d)^2$, then $\sup_{x \in \mathbb{R}^d} |x^{\beta_1} \partial_{\beta_2} f(x)| < \infty$. Moreover, a function $L : (0, \infty) \mapsto (0,\infty)$ is said slowly varying (at zero) if for all $a > 0$, $L(at)/L(t) \to 1$ as $t \to 0$.
	
	Finally, we define algebraically the cumulants. Let $m \geq 1$. For real random variables $X_1, \dots, X_m$, we denote their joint cumulant as:
	\begin{equation}\label{def_kappa}
		\kappa(X_1, \dots, X_m) := \sum_{n = 1}^m (-1)^{n-1}(n-1)! \sum_{B_1, \dots, B_n}\prod_{i = 1}^n \mathbb{E}\left[\prod_{b \in B_i} X_{b}\right],
	\end{equation}
	where $\sum_{B_1, \dots, B_n}$ denotes the sums over all the partitions of $\{1, \dots, m\}$ into $n$ sets $B_1,\dots, B_n$.
	Moreover, for a real random variable $X$ having moments of order $m$, its $m$-th cumulant is defined as
	\begin{equation}\label{eq_kappa_m}
		\kappa_m(X) = \kappa(X, \dots, X),
	\end{equation}
	where $X$ appears $m$-times.

	\section{Cumulants of smooth linear statistics of perturbed lattices}\label{sec_exp_var}
	
	The core of the proofs concerning points~\ref{i_d3_thm_syn}-\ref{i_11_thm_syn} of Theorem~\ref{thm_synthese} is the following Proposition~\ref{prop_kappa}, which derives a simple asymptotic expression for the cumulants of the linear statistics $T_r^0$ associated to the non-stationary lattice $\Phi^0$ (see \eqref{eq:def_phi}), as $r \to \infty$. We deduce from this proposition the main term of the asymptotic expansion of $\mathbb{E}[T_r^0]$ and $\operatorname{Var}[T_r^0]$ in Corollary~\ref{cor_expec_var}, Proposition~\ref{prop_var} and Lemma~\ref{lemma_var_d3}.
	
	The key argument of the proof of Proposition~\ref{prop_kappa} is Lemma~\ref{lemma_poisson_k_0}, which allows us to express the cumulants of $T_r^0$ in the Fourier domain.
	
	\medskip \begin{lemma}\label{lemma_poisson_k_0}
		Let $n \geq 1$, $(f_i)_{i = 1 \dots, n}$ be a family of Schwartz functions and $\xi$ be a random variable with characteristic function $\varphi$. Then, the following equality holds, where both terms are finite:
		\begin{equation}\label{eq_lemma_poisson}
			\sum_{x \in \mathbb{Z}^d} \prod_{i = 1}^n \mathbb{E}[f_i(x+\xi)] =  \sum_{x \in \mathbb{Z}^d} \circledast_{i = 1}^n \{\overline{\mathcal{F}[f_i]} \varphi\}(x).
		\end{equation}
	\end{lemma}
	
	Lemma~\ref{lemma_tail_0} simplifies the right hand side of the above sum by proving that all terms except one are negligible when considering $f_i(\cdot/r)$ instead of $f_i$ and letting $r \to \infty$. 
	
	\medskip \begin{lemma}\label{lemma_tail_0}
		Let $r > 0$, $n \geq 1$, $(f_i)_{i = 1 \dots, n}$ be a family of Schwartz functions and $\varphi$ be a characteristic function. We denote $f_{i, r} = f_i(\cdot/r)$. Then, for all $\beta > 0$, there exists a constant $C(\beta) < \infty$ such that:
		$$\sum_{x \in \mathbb{Z}^d\setminus\{0\}} \left|\circledast_{i = 1}^n \{\overline{\mathcal{F}[f_{i, r}]}\varphi\}(x)\right| \leq C(\beta) r^{-\beta}.$$
	\end{lemma}
 	In order to focus directly on the applications of these two crucial lemmas in Sections~\ref{sec_exp_var}-\ref{sec_stat}, their proofs, whose milestone is the Poisson summation formula, are deferred to Section~\ref{sec_technical}. More precisely, they follow as corollaries of the slightly more general Lemmas~\ref{lemma_poisson_k} and~\ref{lemma_tail}.
	
	Gathering Lemma~\ref{lemma_poisson_k_0} and~\ref{lemma_tail_0}, we obtain in the next proposition a convenient asymptotic expansion of the cumulants $\kappa_m(T_r^0)$ as $r \to \infty$ (recall the definition~\eqref{eq_kappa_m} of the cumulants of a real random variable). 
	
	\medskip \begin{proposition}\label{prop_kappa}
		Let $d \geq 1$, $f \in \mathcal{S}(\mathbb{R}^d)$ and  $(\xi_x)_{x \in \mathbb{Z}^d}$ be i.i.d. random variables with characteristic function $\varphi$. We denote $f_r = f(\cdot/r)$. Let $m \geq 1$. Then, for all $ \beta > 0$:
		$$\kappa_m\left(T_r^0\right) = \sum_{n = 1}^m (-1)^{n-1} (n-1)! \sum_{B_1, \dots, B_n}  \circledast_{i = 1}^n \{\overline{\mathcal{F}[f_r^{|B_i|}]}\varphi\}(0) + O(r^{-\beta}).$$
	\end{proposition}
	\begin{proof}
		The quantity that we are interested in is $\kappa_m(T_r^0) = \kappa_m\left(\sum_{x \in \mathbb{Z}^d} f\left((x + \xi_x)/r\right)\right)$. Since the cumulants are multi-linear, and using  the Fubini's theorem for the sums overs $\mathbb{Z}^{dk}$ for $1 \leq k \leq m$, justified by Lemma~\ref{lemma_moment}, we obtain,
		\begin{align*}
			\kappa_m(T_r^0) = \sum_{x_1, \dots, x_m \in \mathbb{Z}^{dm}} \kappa\left(f\left(\frac{x_1+\xi_{x_1}}r\right), \dots, f\left(\frac{x_m+\xi_{x_m}}r\right)\right).
		\end{align*}
		Then, this expression can be simplified further using the independence of the perturbations, 
		\begin{align*}
			\kappa_m(T_r^0) = \sum_{x \in \mathbb{Z}^d} \kappa_m\left(f\left(\frac{x+\xi_x}r\right)\right).
		\end{align*}
		Let $\xi$ be a random variable with characteristic function $\varphi$. Recalling the definition~\eqref{def_kappa} of the cumulants $\kappa_m$ and using the Poisson summation formula of Lemma~\ref{lemma_poisson_k_0}, we obtain
		\begin{align}\label{eq_kappa_int}
			\kappa_m(T_r^0) &=\sum_{n = 1}^m (-1)^{n-1} (n-1)! \sum_{B_1, \dots, B_n}   \sum_{x \in \mathbb{Z}^d} \prod_{i = 1}^n \mathbb{E}\left[f_r^{|B_i|}\left(x + \xi\right)\right] \nonumber\\& =\sum_{n = 1}^m (-1)^{n-1} (n-1)! \sum_{B_1, \dots, B_n}   \sum_{x \in \mathbb{Z}^d}\circledast_{i = 1}^n \{\overline{\mathcal{F}[f_r^{|B_i|}]}\varphi\}(x) .
		\end{align}
		To conclude, we extract the term $x = 0$ in the previous series and use Lemma~\ref{lemma_tail_0}.
	\end{proof}
	
	The next corollary is simply the previous proposition when $m \in \{1, 2\}$. Firstly, it shows that, although $\Phi^0$ is not necessarily stationary, the expectation of $T_r^0$ is the one of stationary perturbed lattice $\Phi^{1}$ (according to Campbell’s averaging formula~\cite{daley2003introduction} or Lemma~\ref{lemma_exp_exp_stat}), plus a remainder term that vanishes quickly when considering the limit $r \to \infty$. Secondly, it provides a convenient asymptotic expression for the variance of $T_r^0$. Similarly, although we are not necessarily considering stationary point processes with simple points, the variance behaves asymptotically similarly and $1 - |\varphi|^2$ can be interpreted as the structure factor of $\Phi^0$~\cite{torquato2018hyperuniform, kim2018effect, hawat2023estimating}, also called density of the Bartlett spectral measure~\cite{daley2003introduction}. As highlighted in the proof of the previous proposition, the benefit to consider non-stationary perturbed lattices is the possibility to use the independence between the random variables $(\xi_x)_{x \in \mathbb{Z}^d}$ to reduce the $dm$-dimensional sums appearing in the $m$-th cumulants to $d$-dimensional ones.  
	
	\medskip \begin{corollary}\label{cor_expec_var}
		Let $d \geq 1$ and $f \in \mathcal{S}(\mathbb{R}^d)$. We consider a family of i.i.d. random variables $(\xi_x)_{x \in \mathbb{Z}^d}$ with characteristic function $\varphi$. Then, for all $\beta > 0$.
		\begin{equation}\label{eq_exp}
			\mathbb{E}[T_r^0] \underset{r \to \infty}{=} r^d \int_{\mathbb{R}^d} f(x) dx + O(r^{-\beta}),
		\end{equation}
		\begin{equation}\label{eq_var}
			\operatorname{Var}[T_r^0] = r^d \int_{\mathbb{R}^d} |\mathcal{F}[f](x)|^2 (1 - |\varphi(x/r)|^2)dx + O(r^{-\beta}).
		\end{equation}
	\end{corollary}
	\begin{proof}
		Equation~\eqref{eq_exp} is directly given by Proposition~\ref{prop_kappa} with $m = 1$. Concerning equation~\eqref{eq_var}, according to Proposition~\ref{prop_kappa} for $m = 2$, we have
		\begin{align*}
			\operatorname{Var}[T_r^0] &= \mathcal{F}[f_r^2](0) - \{\overline{\mathcal{F}[f_r]} \varphi\} \ast \{\overline{\mathcal{F}[f_r]} \varphi\}(0) + O(r^{-\beta}),
		\end{align*} 
		where $f_r(x) = f(x/r)$. Finally,~\eqref{eq_var} follows from the Plancherel theorem:
		\begin{align*}
			\operatorname{Var}[T_r^0] &=  r^d\int_{\mathbb{R}^d} |\mathcal{F}[f](x)|^2 dx -r^d\int_{\mathbb{R}^d} |\mathcal{F}[f](rx)|^2 |\varphi(x)|^2dx + O(r^{-\beta}).
		\end{align*}
	\end{proof}

	With equation~\eqref{eq_var}, we deduce in the next proposition the asymptotic behavior of the variance of $T_r^0$ when the characteristic function of the perturbations possesses a prescribed behavior near zero. Although this result will be not used before Section~\ref{sec_alpha_2}, it motivates the plan of the paper, as detailed in Remark~\ref{rmk_alpha} below. Note that its proof mimics the one of Theorem 3 of~\cite{soshnikov2002gaussian}, who considers determinantal point processes. 
	
	\medskip \begin{proposition}\label{prop_var}
		Let $d \geq 1$ and $f \in \mathcal{S}(\mathbb{R}^d)$. We consider i.i.d. random variables $(\xi_x)_{x \in \mathbb{Z}^d}$ with characteristic function $\varphi$ and suppose that $$1 - |\varphi(x)|^2 \sim L(|x|) |x|^{\alpha},$$ as $|x| \to 0$, where $\alpha \in (0, 2]$ and $L$ is slowly varying. Then,
		$$\operatorname{Var}[T_r^0] \underset{r \to \infty}{\sim} L(1/r) r^{d-\alpha}\int_{\mathbb{R}^d} |\mathcal{F}[f](x)|^2|x|^{\alpha} dx.$$ 
	\end{proposition}
	\begin{proof}
		We use the expression of the variance of $T_r^0$ given by Proposition~\ref{prop_kappa}, and choose $\beta = \alpha$. Accordingly,
		\begin{align*}
			\frac{r^{\alpha-d}}{L(1/r)}\operatorname{Var}[T_r^0] &= \frac{r^{\alpha}}{L(1/r)}\left( \int_{\mathbb{R}^d} |\mathcal{F}[f](x)|^2 (1 -|\varphi(x)|^2)dx\right) +  o(1) \\& =  \int_{\mathbb{R}^d} |\mathcal{F}[f](x)|^2 \frac{r^{\alpha}}{L(1/r)}(1 - |\varphi(x/r)|^2)dx + o(1).
		\end{align*}
		Let $\varepsilon > 0$. The assumption on the behavior of $\varphi$ near zero ensures that there exists $\delta > 0$ such that for all $|y| \leq \delta$, then $|1 - |\varphi(y)|^2 - L(|y|)|y|^{\alpha}| \leq \varepsilon| L(|y|)| |y|^{\alpha}$. This implies
		\begin{align*}
			\left|\frac{r^{\alpha-d}}{L(1/r)}\operatorname{Var}[T_r^0] - \int_{\mathbb{R}^d} |\mathcal{F}[f](x)|^2|x|^{\alpha} dx\right| \leq A + B+ D + o(1),
		\end{align*}
		where 
		\begin{align*}
			&A := \varepsilon \int_{B(0, \delta r)} |\mathcal{F}[f](x)|^2 |x|^{\alpha} \frac{L(|x|/r)}{L(1/r)} dx,
			\\&B := \int_{B(0, \delta r)} |\mathcal{F}[f](x)|^2 |x|^{\alpha} \left|\frac{L(|x|/r)}{L(1/r)} -1\right| dx, 
			\\& D := \frac{2r^{\alpha}}{L(1/r)} \int_{B(0, \delta r)^c} |\mathcal{F}[f](x)|^2 dx + \int_{B(0, \delta r)^c} |\mathcal{F}[f](x)|^2 |x|^{\alpha} dx.
		\end{align*}
		Since $L$ is slowly varying, there exists $(b, C, q, r_0) \in (0, \infty)^4$ such that for all $r \geq r_0$ and $0 < |x| \leq b r$, we have $L(|x|/r)/L(1/r) \leq C \max(|x|^q, |x|^{-1/2})$ (see Proposition~\ref{prop_slwy}). Subsequently, by reducing $\delta$ if necessary to ensure $\delta < b$, we obtain, for $r \geq r_0$,
		$$A \leq \varepsilon C \int_{B(0, \delta r)} |\mathcal{F}[f](x)|^2 |x|^{\alpha} \max(|x|^q, |x|^{-1/2}) dx =: \varepsilon C'.$$
		Note that $C' < \infty$ as $f \in \mathcal{S}(\mathbb{R}^d)$ and $1/2 < d$. Moreover, according to the Lebesgue's dominated convergence theorem $B$ converges to $0$. Indeed, $L(|x|/r)/L(1/r)$ converges to $1$ for $x \neq 0$, and the integrand is upper bounded by
		\begin{align*}
			\mathbf{1}_{|x| \leq br}|\mathcal{F}[f](x)|^2 |x|^{\alpha} \left|\frac{L(|x|/r)}{L(1/r)} -1\right| \leq C |\mathcal{F}[f](x)|^2  |x|^{\alpha} (\max(|x|^q, |x|^{-1/2})+1),
		\end{align*}
		which is $L^1(\mathbb{R}^d)$ since $f \in \mathcal{S}(\mathbb{R}^d)$ and $1/2 < d$. Finally, $D$ converges to $0$ as $r \to \infty$ because $f \in \mathcal{S}(\mathbb{R}^d)$. Accordingly, we have obtained that for all $\varepsilon > 0$, 
		\begin{align*}
			\underset{r \to \infty}{\operatorname{\lim\sup}}\left|\frac{r^{\alpha-d}}{L(1/r)}\operatorname{Var}[T_r^0] - \int_{\mathbb{R}^d} |\mathcal{F}[f](x)|^2|x|^{\alpha} dx\right| \leq \varepsilon C'.
		\end{align*}
		Finally, taking the limit as $\varepsilon \to 0$ gives the result.
	\end{proof}
	
	\medskip \begin{remark}\label{rmk_alpha}
		The previous proposition implies that the variance of $T_r^{0}$ diverges to infinity when $\alpha < d$. Since the parameter $\alpha$ is necessarily smaller than 2 (see Proposition~\ref{prop_lim_al}), it suggests that $\operatorname{Var}[T_r^0]$ should always diverges when $d \geq 3$ with a rate of order at least $r^{d-2} \geq r$. We prove this result in Lemma~\ref{lemma_var_d3} from which the asymptotic normality follows in Theorem~\ref{thm_tcl_var_unbouded}. When $\alpha = d$ and $L$ is bounded, the variance does not diverge as $r \to \infty$ and the previous proof has to be refined (see Theorems~\ref{thm_tcl_var_bouded} and~\ref{thm_no_clt_alpha_1}). Finally, when the variance converges to $0$, i.e. when $\alpha > d$, the previous proofs using the cumulants fail and one has to use different arguments (see Section~\ref{sec_no_clt}). 
	\end{remark}
	
	\medskip \begin{remark}
		The asymptotic scaling in the previous proposition is characteristic of hyperuniform point processes, which are generally classified according to the value of $\alpha$: class~I if $\alpha > 1$, class II if $\alpha = 1$ and class III if $\alpha \in (0, 1)$~\cite{torquato2018hyperuniform}. This classification is explained by the asymptotic behavior of the number variance $\operatorname{Var}\left[\sum_{x \in \Phi^0} \mathbf{1}_{B(0, 1)}(x/r)\right]$ when the slowly varying function $L$ is constant. Indeed, due to the slow decay of the Fourier transform of $\mathbf{1}_{B(0, 1)}$, the number variance scales for large $r$ as $r^{d-1} \log(r)$ for Class II hyperuniform point processes and scales as $r^{d-1}$ for the ones of Class I (which is larger than the rate $r^{d-\alpha}$ when considering smooth functions). 
	\end{remark}
	
	\section{Gaussian limits}\label{sec_gauss}
	
	\subsection{Central limit theorem in dimension $d \geq 3$}\label{sec_clt_inf_d}
	
	In this section, we prove a central limit theorem when the dimension is larger than three, assuming only the non-degeneracy of the perturbations. The proof is based on the method of cumulants and follows from the lemma below. It extends the observations of Remark \ref{rmk_alpha} to any non a.s. constant perturbations: the variance of $T_r^0$ always diverges at rate of order at least $r^{d-2}$.
	
	\medskip \begin{lemma}\label{lemma_var_d3}
		Let $d \geq 1$ and $f \in \mathcal{S}(\mathbb{R}^d)$. We assume that there exists $x_0 \in \mathbb{R}^d$ such that $f(x_0) \neq 0$. We consider i.i.d. random variables $(\xi_x)_{x \in \mathbb{Z}^d}$ that are not a.s. constant. Let $h : (0, \infty) \to (0, \infty)$ satisfying $\lim_{r \to \infty} h(r) = 0$. Then, there exist $C > 0$, and $r_0> 0$ such that for all $r > r_0$:
		$$\operatorname{Var}[T_r^0] \geq C r^{d-2} h(r).$$ 
	\end{lemma}
	\begin{proof}
		We assume that for all $C > 0$ and $r_0>0$ there exists $r > r_0$ such that $\operatorname{Var}[T_r^0] < C r^{d-2}h(r)$, and seek a contradiction. Accordingly, there exists a sequence $(r_n)_{n \geq 1}$ with $\lim r_n = \infty$ such that for all $n \geq 1$, $\operatorname{Var}[T_{r_n}^0] \leq r_n^{d-2} h(r_n)$. Using equation~\eqref{eq_var} with $\beta = 2$, we get
		\begin{equation*}
			\int_{\mathbb{R}^d} |\mathcal{F}[f](x)|^2 r_n^{2}(1 -|\varphi(x/r_n)|^2) dx \leq C h(r_n) + O(r_n^{-d}).
		\end{equation*}
		Since $1 -|\varphi(x/r_n)|^2 \geq 0$, the Fatou's lemma yields
		\begin{equation}\label{eq_fatou1}
			0 \leq \int_{\mathbb{R}^d} |\mathcal{F}[f](x)|^2 \underset{n \to \infty}{\operatorname{\lim\inf}}~r_n^{2}(1 -|\varphi(x/r_n)|^2) dx \leq 0.
		\end{equation}
		To express $~r_n^{2}(1 -|\varphi(x/r_n)|^2)$ in a convenient way, we introduce $\xi'$ and $\xi$ that are independent random variables with the the same distribution as $\xi_0$. Then, using the independence between $\xi$ and $\xi'$, we obtain
		\begin{align}\label{eq_phi_couplage}
			1 -|\varphi(x/r_n)|^2 &= \Re\left\{1 - \mathbb{E}\left[e^{2\bm{i}\pi (\xi-\xi') \cdot x/r_n}\right]\right\} = 2\mathbb{E}\left[\sin^2\left(\frac{(\xi - \xi') \cdot x}{2r_n}\right)\right].
		\end{align}
		The next step is to use again the Fatou's lemma, to get
		\begin{align*}\label{eq_fatou2}
			\underset{n \to \infty}{\operatorname{\lim\inf}}~ r_n^{2}(1 -|\varphi(x/r_n)|^2) dx & \geq 2 \mathbb{E}\left[\underset{n \to \infty}{\operatorname{\lim\inf}} r_n^2\sin^2\left(\frac{(\xi - \xi') \cdot x}{2r_n}\right)\right] \\& = \frac12\mathbb{E}\left[(x\cdot\xi - x\cdot\xi')^2\right].
		\end{align*}
		The function $f$ is not identically null, so $\mathcal{F}[f]$ takes at least one non zero value at a point $x_1 \in \mathbb{R}^d$. By continuity, there exists $r_1> 0$ such that for all $x \in B(x_1, r_1)$, we have $|\mathcal{F}[f](x)| > 0$. Accordingly, using~\eqref{eq_fatou1} and the previous inequality, we obtain:
		\begin{equation}\label{eq_lower_bound_var_forall}
			\forall x \in B(x_1, r_1),~x\cdot\xi = x\cdot\xi',~a.s.
		\end{equation}
		Let $(e_i)_{1 \leq i \leq d}$ be the canonical Euclidean basis of $\mathbb{R}^d$. Using~\eqref{eq_lower_bound_var_forall} with $x = x_1 + r_1 e_i/2$ for $i \in \{1, \dots, d\}$, we thus obtain $\xi = \xi'$ a.s. To conclude, recall that $\xi$ and $\xi'$ are independent. Subsequently, $\xi$ is a.s. constant, which is a contradiction and finishes the proof.
	\end{proof}
	
	The previous lemma was the last tool needed to prove the following central limit theorem in dimension $d \geq 3$. Its proof relies on the expression of the cumulants of $T_r^0$ obtained in Proposition~\ref{prop_kappa}, that is used as a Brillinger-mixing condition~\cite{heinrich2016strong, biscio2016brillinger}.

	\medskip \begin{theorem}\label{thm_tcl_var_unbouded}
		Let $d \geq 3$ and $f \in \mathcal{S}(\mathbb{R}^d)$. We assume that there exists $x_0 \in \mathbb{R}^d$ such that $f(x_0) \neq 0$. We consider i.i.d. random variables $(\xi_x)_{x \in \mathbb{Z}^d}$ that are not a.s. constant. Then:
		$$\operatorname{Var}[T_r^0]^{-1/2} \left(T_r^0 - r^d \int_{\mathbb{R}^d} f(x) dx \right)\xrightarrow[r \to \infty]{d} \mathcal{N}\left(0, 1\right).$$
	\end{theorem}
	\begin{proof}
		We denote $Y_r := \operatorname{Var}[T_r^0]^{-1/2}(T_r^0 - r^d \int_{\mathbb{R}^d} f(x)dx)$. Equation~\eqref{eq_exp} of Corollary~\ref{cor_expec_var} ensures that its first cumulant converges to $0$. Its second cumulant equals $1$. Concerning the higher order cumulants, according to Proposition~\ref{prop_kappa}:
		$$\forall m \geq 3,~\kappa_m\left(T_r^0\right) = \sum_{n = 1}^m (-1)^{n-1} (n-1)! \sum_{B_1, \dots, B_n}  \circledast_{i = 1}^n \{\overline{\mathcal{F}[f_r^{|B_i|}]}\varphi\}(0) + O(r^{-1}).$$
		To bounds these cumulant, we use the Young's inequality for the convolution and the fact that $|\varphi| \leq 1$. Indeed, for $n \geq 2$, we have
		\begin{align*}
			\left|\circledast_{i = 1}^n\{\overline{\mathcal{F}[f_r^{|B_i|}]}\varphi\}(0)\right| &\leq \|\mathcal{F}[f_r^{|B_1|}] \varphi\|_{\infty} \prod_{i = 2}^n \|\mathcal{F}[f_r^{|B_i|}]  \varphi\|_1 \\& \leq r^d \|\mathcal{F}[f^{|B_1|}]\|_{\infty} \prod_{i = 2}^n \|\mathcal{F}[f^{|B_i|}]\|_1.
		\end{align*}
		Moreover, a similar bound holds for $n = 1$, 
		$$\left|\circledast_{i = 1}^1\{\overline{\mathcal{F}[f_r^{|B_i|}]}\varphi\}(0)\right| = \left|\mathcal{F}[f_r^{|B_i|}](0)\right|\leq r^d \|f^{|B_i|}\|_1.$$
		Gathering both previous inequalities, we have proven that for all $m \geq 3$ there exists a constant $c_m < \infty$, that does not depend on $r$, such that
		\begin{equation}\label{eq_bound_kappa_m}
			\forall r > 0, |\kappa_m(T_r^0)| \leq c_m r^d.
		\end{equation}
		The last step is to lower bound the variance. To do so, we use Corollary~\ref{lemma_var_d3} with $h(r) = r^{-1/2}$. Subsequently, there exist $C, r_0> 0$ such that for all $r \geq r_0$,  $$\operatorname{Var}[T_r^0] \geq C r^{d-2} r^{-1/2} \geq C r^{1/2}.$$ Finally, for $m \geq 3$, we have,  
		$$\left|\kappa_m(\operatorname{Var}[T_r^0]^{-1/2} T_r^0)\right| \leq C_m r^{d - m/4} +O(r^{-3/2}),$$ 
		for some $C_m < \infty$. Since the right hand side of the previous inequality converges to $0$ as $r\to \infty$ when $m > 4d$, this concludes the proof recalling the classical result of Marcinkiewicz (see, e.g. Lemma~3 of~\cite{soshnikov2002gaussian}).
	\end{proof}
	
	\subsection{Central limit theorem in dimension $d = 2$}\label{sec_alpha_2}
	
	In dimension smaller than 2 the variance of $T_r^0$ can be bounded and we need a finer control of the cumulants than the one used in the proof of Theorem~\ref{thm_tcl_var_unbouded}. In Section~\ref{sec:clt_l2}, we prove a central limit theorem with a bounded rate of convergence, a case that occurs when the perturbations admit a moment of order $2$ (see Proposition \ref{prop_var_bounded}). Then, in Section~\ref{sec:clt_ht}, we examine the case of general perturbations with unbounded variance in dimension $2$. We establish a central limit theorem for a subsequence of $T_r^0$. The extension to the full sequence is addressed in Section~\ref{sec:clt_rv_phi} with a mild assumption on the behavior of the characteristic function of the perturbations near zero.
	
	\subsubsection{The case of square integrable perturbations}\label{sec:clt_l2}
	
	To prove normality results using the method of cumulants when the considered sequence of random variables has an asymptotically bounded variance, one may rely on asymptotic cancellations between terms. These simplifications, of combinatoric nature, are ensured by Lemma~\ref{lemma_cmb}. Note that equation~\eqref{eq_cmb_1} could be proved easily by noticing that it equals $\kappa_m(\bm{1})$ where $\bm{1}$ denotes the random variable a.s. equal to $1$. However, in order to prove equations~\eqref{eq_cmb_2} and~\eqref{eq_cmb_3}, we recall a general procedure, using algebraic cumulants, defined as the coefficients of the composition $\log(h(x))$ of the two (formal) entire series $\log$ and $h$. 
	\medskip \begin{lemma}\label{lemma_cmb}
		For all $m \geq 3$, 
		\begin{equation}\label{eq_cmb_1}
			\sum_{n = 1}^m (-1)^{n-1} (n-1)! \sum_{B_1, \dots, B_n} 1 = 0,
		\end{equation}
		\begin{equation}\label{eq_cmb_2}
			\sum_{n = 1}^m (-1)^{n-1} (n-1)! \sum_{B_1, \dots, B_n}  \sum_{i = 1}^n |B_i|^2 = 0,
		\end{equation}
		\begin{equation}\label{eq_cmb_3}
				\sum_{n = 1}^m (-1)^{n-1} (n-1)! \sum_{B_1, \dots, B_n} \sum_{i \neq j} |B_i||B_j| = 0,
		\end{equation}
	where we recall that $\sum_{B_1, \dots, B_n}$ denotes the sum over all the partitions of $\{1, \dots, m\}$ into $n$ sets $B_1,\dots, B_n$ (see Section \ref{sec_notations}).
	\end{lemma}
	\begin{proof}
		The proof of this result is done in Lemma 9 of~\cite{rider2006complex} but with slightly different notations. For completeness, we recall the arguments, that are based on Faà-Di-Bruno formula. Let $h(x) = \sum_{k \geq 1} \frac{a_k}{k!} x^k$ be an entire series. Then, when it is well defined, the coefficients $c_m$ of the entire series $\log(h(x))$ can be computed~\cite{Stanley_Fomin_1999}:
		$$\forall m \geq 1,~c_m = \sum_{n = 1}^m (-1)^{n-1} (n-1)! \sum_{B_1, \dots, B_n} \prod_{i = 1}^n a_{|B_i|}.$$
		To exploit this relation, we consider $t \in (0, 1)$ and choose $a_k = t k +1$, that corresponds to $h(x) = e^{x}(tx+1)$. As a consequence, we obtain the entire series expansion of $\log(h(x))$:
		$$\forall |x| < t^{-1}, \log(h(x)) = x + \sum_{m = 1}^{\infty} \frac{(-1)^{m-1} t^m x^m}{m}.$$
		By uniqueness of the coefficients of this entire series, we have, for all $m \geq 2$, and $t \in (0, 1)$:
		$$\sum_{n = 1}^m (-1)^{n-1} (n-1)! \sum_{B_1, \dots, B_n} \prod_{i = 1}^n (t |B_i|+1) =  \frac{(-1)^{m-1}}{m} t^m.$$
		By identifying the constant coefficient of the previous polynomials in $t$, we obtain~\eqref{eq_cmb_1}. When $m \geq 3$, by identifying the coefficient of $t^2$ of both sides, we get~\eqref{eq_cmb_3}:
		\begin{equation*}
			\sum_{n = 1}^m (-1)^{n-1} (n-1)! \sum_{B_1, \dots, B_n} \sum_{i \neq j} |B_i||B_j| = 0.
		\end{equation*}
		Finally, using~\eqref{eq_cmb_1} and~\eqref{eq_cmb_3}, we obtain~\eqref{eq_cmb_2}.
	\end{proof}
	
	The next theorem states that in dimension $d = 2$ and for perturbations having a moment of order $2$, a central limit theorem is satisfied by $T_r^0$, even if its variance is asymptotically bounded. Its proof consists in the Taylor expansion at order 2 of the characteristic function of the perturbations. This allows to compute the limits of the cumulants of $T_r^0$ of all orders $m \geq 1$. Then using the previous lemma we show that all limits for $m \geq 3$ are in fact zero.
	
	\medskip \begin{theorem}\label{thm_tcl_var_bouded}
		We assume that $d = 2$ and consider i.i.d. random variables $(\xi_x)_{x \in \mathbb{Z}^2}$, $\xi$ and $\xi'$. We suppose that $\mathbb{E}[|\xi_0|^2]< \infty$. Let $f \in \mathcal{S}(\mathbb{R}^2)$. Then:
		$$T_r^0 - r^2\int_{\mathbb{R}^2} f(x) dx \xrightarrow[r \to \infty]{d} \mathcal{N}\left(0, \frac12\int_{\mathbb{R}^2} |\mathcal{F}[f](x)|^2 \mathbb{E}[|(\xi - \xi') \cdot x|^2] dx\right).$$
	\end{theorem}
	\begin{proof}
		Equation~\eqref{eq_exp} of Corollary~\ref{cor_expec_var} ensures that the first cumulant vanishes as $r \to \infty$: $$\lim_{r \to \infty} \kappa_1(T_r^0 - r^2\int_{\mathbb{R}^2} f(x) dx) = 0.$$
		Concerning the second order cumulant, according to equation~\eqref{eq_var} of Corollary~\ref{cor_expec_var} and the computation~\eqref{eq_phi_couplage}, we have
		$$\kappa_2(T_r^0) = 2 \int_{\mathbb{R}^2} |\mathcal{F}[f](x)|^2\mathbb{E}\left[r^2\sin^2\left(\frac{(\xi - \xi') \cdot x}{2r}\right)\right] dx.$$
		Moreover, 
		$$|\mathcal{F}[f](x)|^2r^2\sin^2\left(\frac{(\xi - \xi') \cdot x}{2r}\right) \leq \frac14|\mathcal{F}[f](x)|^2 |x|^2 |\xi - \xi'|^2,$$
		which is integrable with respect to the measure $dx \otimes d\mathbb{P}$ since $\xi$ and $\xi'$ are square integrable. By the Lebesgue's dominated convergence theorem, we thus obtain
		\begin{equation}\label{eq_lim_k2}
			\lim_{r \to \infty} \kappa_2(T_r^0) =  \frac12 \int_{\mathbb{R}^2} |\mathcal{F}[f](x)|^2\mathbb{E}\left[|(\xi - \xi') \cdot x|^2\right] dx.
		\end{equation}
		It remains to prove that for all $m \geq 3$, $\kappa_m(T_r^0) \to 0$ as $r \to \infty$. According to Proposition~\ref{prop_kappa}, we can express these cumulants as
		$$\kappa_m(T_r^0) =\sum_{n = 1}^m (-1)^{n-1} (n-1)! \sum_{B_1, \dots, B_n}  \circledast_{i = 1}^n \{\overline{\mathcal{F}[f_r^{|B_i|}]}\varphi\}(0) + o(1),$$
		where $f_r(x)= f(x/r)$ and $\varphi_r(x) = \varphi(x/r)$. For $n \geq 2$, using the changes of variable $x_i \leftrightarrow r x_i$, we have
		$$\circledast_{i = 1}^n\{\overline{\mathcal{F}[f_r^{|B_i|}]}\varphi\}(0) =r^2 \circledast_{i = 1}^n\{\overline{\mathcal{F}[f_r^{|B_i|}]}\varphi_r\}(0).$$
		The next step is to write $\varphi(x/r) = 1+(\varphi(x/r) - 1)$ and to expand the product over $i$:
		\begin{align*}
			 \prod_{i = 1}^n &\varphi\left(\frac{x_i - x_{i-1}}r\right) = 1+ \sum_{q = 1}^n  \sum_{S_q} \prod_{s \in S_q} \left(\varphi\left(\frac{x_s - x_{s-1}}r\right)-1\right),
		\end{align*}
		where $\sum_{S_q}$ denotes the sum over the sets $S_q$ of $\{1, \dots, n\}$ of cardinality $|S_q| = q$. In the following, we denote $\mu_i = 1$ if $i \in S_q$ and $0$ otherwise and use the notation $(\varphi_r -1)^{0} = 1$. Then, the previous expansion and the Fourier inversion theorem give that for $n \geq 2$, 
		\begin{align*}
			\circledast_{i = 1}^n \{\overline{\mathcal{F}[f_r^{|B_i|}]}\varphi_r\}(0) &=\circledast_{i = 1}^n \{\mathcal{F}[f^{|B_i|}]\}(0)  + \sum_{q = 1}^n  \sum_{S_q} \circledast_{i = 1}^n \{\overline{\mathcal{F}[f_r^{|B_i|}]}(\varphi_r-1)^{\mu_i}\}(0) \\& = \int_{\mathbb{R}^2} f^m(x) dx + \sum_{q = 1}^n  \sum_{S_q}\circledast_{i = 1}^n \{\overline{\mathcal{F}[f_r^{|B_i|}]}(\varphi_r-1)^{\mu_i}\}(0).
		\end{align*}
		Using equation~\eqref{eq_cmb_1} of Lemma~\ref{lemma_cmb} we get the next intermediate expression for the cumulants:
		\begin{align}\label{eq_kappa_m_Tr_non_asymp}
			&\kappa_m(T_r^0) = \mathcal{F}[f_r^{m}](0)\varphi(0) + r^2\sum_{n = 2}^m (-1)^{n-1} (n-1)! \sum_{B_1, \dots, B_n}  \int_{\mathbb{R}^2} f^m(x) dx \nonumber \\& \qquad+ r^2\sum_{n = 2}^m (-1)^{n-1} (n-1)! \sum_{B_1, \dots, B_n}  \sum_{q = 1}^n  \sum_{S_q} \circledast_{i = 1}^n \{\overline{\mathcal{F}[f_r^{|B_i|}]}(\varphi_r-1)^{\mu_i}\}(0) \nonumber +o(1)\\& = r^2\sum_{n =2}^m (-1)^{n-1} (n-1)! \sum_{B_1, \dots, B_n}  \sum_{q = 1}^n  \sum_{S_q} \circledast_{i = 1}^n \{\overline{\mathcal{F}[f_r^{|B_i|}]}(\varphi_r-1)^{\mu_i}\}(0) +o(1).
		\end{align}
		Now, we compute the limit of~\eqref{eq_kappa_m_Tr_non_asymp}. We begin with the term $q = 1$. For $n \geq 2$, we have
		\begin{align*}
			r^2\sum_{S_1 \subset \{1, \dots, n\}}& \circledast_{i = 1}^n \{\overline{\mathcal{F}[f_r^{|B_i|}]}(\varphi_r-1)^{\mu_i}\}(0) \\&= \sum_{j = 1}^n \int_{(\mathbb{R}^2)^{n-1}} r^2 \left(\varphi\left(\frac{x_j-x_{j-1}}r\right)-1\right)\prod_{i = 1}^n \overline{\mathcal{F}[f_r^{|B_i|}]}(x_i - x_{i-1}) dx \\& = \sum_{j = 1}^n A_j + \sum_{j = 1}^n \tilde{A_j},
		\end{align*}
		with the convention $x_0 = x_n = 0$ and where 
		\begin{align*}
			&A_j = \int_{(\mathbb{R}^2)^{n-1}} \left(\varphi\left(\frac{x_j-x_{j-1}}r\right)-1 - \nabla \varphi(0) \cdot\frac{x_j-x_{j-1}}r \right)\prod_{i = 1}^n \overline{\mathcal{F}[f_r^{|B_i|}]}(x_i - x_{i-1}) dx,
			\\& \tilde{A_j} = r \int_{(\mathbb{R}^2)^{n-1}} \nabla \varphi(0) \cdot(x_j-x_{j-1})\prod_{i = 1}^n \overline{\mathcal{F}[f_r^{|B_i|}]}(x_i - x_{i-1}) dx.
		\end{align*}
		The terms $\tilde{A_j}$ will be considered later in the proof. For now, we compute the limit of the terms $A_j$. Since the perturbations have a moment of order $2$, $\varphi$ is two times differentiable. As a consequence, we can pass to the limit with the Lebesgue's dominated convergence theorem, justified by the fact that the functions $(f^{|B_i|})_{1 \leq i \leq n}$ are Schwartz:
		\begin{align}\label{eq_lim_aj}
			\lim_{r \to \infty} A_j = \frac12\int_{(\mathbb{R}^2)^{n-1}} (x_j-x_{j-1}) \cdot \nabla^2 \varphi(0) (x_j-x_{j-1})\prod_{i = 1}^n \overline{\mathcal{F}[f_r^{|B_i|}]}(x_i - x_{i-1}) dx.
		\end{align}
		We compute similarly the limit of the term $q = 2$:
		\begin{align}\label{eq_lim_q2}
			&r^2\sum_{S_2 \subset \{1, \dots, n\}} \circledast_{i = 1}^n \{\mathcal{F}[f_r^{|B_i|}](\varphi-1)^{\mu_i}\}(0) \nonumber= \\& \sum_{j, j' = 1}^n \int r \left(\varphi\left(\frac{x_j-x_{j-1}}r\right)-1\right) r\left(\varphi\left(\frac{x_{j'}-x_{j'-1}}r\right)-1\right)\prod_{i = 1}^n \overline{\mathcal{F}[f_r^{|B_i|}]}(x_i - x_{i-1})dx \nonumber\\& \to \sum_{j, j' = 1}^n \int_{(\mathbb{R}^2)^{n-1}} \nabla \varphi(0) \cdot(x_{j}-x_{j-1}) \nabla \varphi(0) \cdot(x_{j'}-x_{j'-1}) \prod_{i = 1}^n \overline{\mathcal{F}[f_r^{|B_i|}]}(x_i - x_{i-1})dx.
		\end{align}
		Finally, the terms corresponding to $q \geq 3$ converge to $0$ since for $r > 1$, 
		\begin{align}\label{eq_lim_q3}
			&r^2\sum_{S_q} \circledast_{i = 1}^n \left|\{\overline{\mathcal{F}[f_r^{|B_i|}]}(\varphi-1)^{\mu_i}\}(0)\right| \nonumber\\& \leq \sum_{S_q} \int_{(\mathbb{R}^2)^{n-1}} r^2\prod_{s \in S_q} \left|\varphi\left(\frac{x_s - x_{s-1}}r\right)-1\right| \prod_{i = 1}^n \left|\mathcal{F}[f^{|B_i|}](x_i - x_{i-1})\right|dx\nonumber \\& \leq \frac1r\sum_{S_q} \int_{(\mathbb{R}^2)^{n-1}} \prod_{s \in S_q} |\nabla \varphi(0)| |x_s - x_{s-1}| \prod_{i = 1}^n \left|\mathcal{F}[f^{|B_i|}](x_i - x_{i-1})\right|dx.
		\end{align}
		To consider the limit of $\kappa_m(T_r^0)$, it remains to  prove that, 
		\begin{equation}\label{eq_lim_bj}
			\sum_{n = 1}^m (-1)^{n-1} (n-1)! \sum_{B_1, \dots, B_n} \sum_{j = 1}^n \tilde{A_j} = 0.
		\end{equation}
		In fact, $\tilde{A_j} = 0$ for all $j$. Indeed, by denoting $\nabla \varphi(0) = (a, b) \in \mathbb{C}^2$, and using integration by part and the Fourier inversion theorem, we obtain, 
		\begin{align*}
			\tilde{A_j} &= r \int_{(\mathbb{R}^2)^{n-1}} \left(a(x_j^1-x_{j-1}^1) + b (x_j^2-x_{j-1}^2)\right)\prod_{i = 1}^n \overline{\mathcal{F}[f^{|B_i|}]}(x_i - x_{i-1}) dx \\&= - \frac{a}{2\bm{i}\pi} |B_j| \int_{(\mathbb{R}^2)^{n-1}} \overline{\mathcal{F}[\partial_1 f f^{|B_j|-1}]}(x_j - x_{j-1}) \prod_{i = 1, i \neq j}^n \overline{\mathcal{F}[f^{|B_i|}]}(x_i - x_{i-1}) dx \\& \quad - \frac{b}{2\bm{i}\pi } |B_j| \int_{(\mathbb{R}^2)^{n-1}} \overline{\mathcal{F}[\partial_2 f f^{|B_j|-1}]}(x_j - x_{j-1}) \prod_{i = 1, i \neq j}^n \overline{\mathcal{F}[f^{|B_i|}]}(x_i - x_{i-1}) dx \\& = - \frac{|B_j|}{2\bm{i} \pi} \int_{\mathbb{R}^2} (a \partial_1 f(x) + b \partial_1 f(x)) f^{m-1}(x) dx \\& = - \frac{|B_j|}{2\bm{i} \pi} \int_{\mathbb{R}} a \left[f^m(x^1, x^2)\right]_{x^1 = - \infty}^{x^1 = \infty} dx^2  - \frac{|B_j|}{2\bm{i} \pi} \int_{\mathbb{R}} b \left[f^m(x^1, x^2)\right]_{x^2 = - \infty}^{x^2 = \infty} dx^1 = 0.
		\end{align*}
		Finally, gathering~\eqref{eq_lim_aj},~\eqref{eq_lim_q2},~\eqref{eq_lim_q3} and~\eqref{eq_lim_bj}, we obtain the limit of the cumulants, 
		\begin{align*}
			\lim_{r \to \infty} \kappa_m(T_r) = \frac12 \sum_{n =2}^m (-1)^{n-1} (n-1)! \sum_{B_1, \dots, B_n} (C_n + D_n),
		\end{align*}
		where 
		$$C_n= \frac12 \sum_{j  = 1}^n \int_{(\mathbb{R}^2)^{n-1}} (x_j-x_{j-1}) \cdot \nabla^2 \varphi(0) (x_j-x_{j-1})\prod_{i = 1}^n \overline{\mathcal{F}[f_r^{|B_i|}]}(x_i - x_{i-1}) dx,$$
		$$D_n = \sum_{j, j' = 1}^n \int_{(\mathbb{R}^2)^{n-1}} \nabla \varphi(0) \cdot(x_{j}-x_{j-1}) \nabla \varphi(0) \cdot(x_{j'}-x_{j'-1}) \prod_{i = 1}^n \overline{\mathcal{F}[f_r^{|B_i|}]}(x_i - x_{i-1})dx.$$
		The last step is to prove that this limit is zero. We begin with $C_n$ and denote $c = \nabla^2 \varphi(0)_{1, 1}$, $d =\nabla^2 \varphi(0)_{2, 2}$ and $e = \nabla^2 \varphi(0)_{2, 1}$. We have, by integration by part, and with the Fourier inversion theorem,
		\begin{align}\label{eq_Cn}
			2C_n &= \sum_{j  = 1}^n  \int_{(\mathbb{R}^2)^{n-1}} \left(c (x_j^1-x_{j-1}^1)^2 +d(x_j^2-x_{j-1}^2)^2\right) \prod_{i = 1}^n \overline{\mathcal{F}[f^{|B_i|}]}(x_i - x_{i-1}) dx \nonumber\\& \quad + 2 e \sum_{j  = 1}^n  \int_{(\mathbb{R}^2)^{n-1}} (x_j^1-x_{j-1}^1)(x_j^2-x_{j-1}^2) \prod_{i = 1}^n \overline{\mathcal{F}[f^{|B_i|}]}(x_i - x_{i-1}) dx \nonumber\\& = \frac{1}{4\pi^2} \sum_{j  = 1}^n \int_{\mathbb{R}^2} \left(c\partial_1^2 f^{|B_j|}(x) + d \partial_2^2 f^{|B_j|}(x) + 2e \partial_{12}^2 f^{|B_j|}(x) \right)f^{m - |B_j|}(x)dx.
		\end{align}
		Concerning $D_n$, with the already employed notation $\nabla \varphi(0) = (a, b) \in \mathbb{C}^2$, we have
		\begin{align}\label{eq_Dn}
			D_n &= \sum_{j, j' = 1}^n \int_{(\mathbb{R}^2)^{n-1}} \Big(a^2 (x_j^1 - x_{j-1}^1)(x_{j'}^1 - x_{j-1}^1) + 2ab (x_j^1 - x_{j-1}^1)(x_{j'}^2 - x_{j-1}^2)\nonumber \\& \qquad \qquad\qquad \qquad \qquad + b^2 (x_j^2 - x_{j-1}^2)(x_{j'}^2 - x_{j-1}^2)\Big)\prod_{i = 1}^n \overline{\mathcal{F}[f^{|B_i|}]}(x_i - x_{i-1})dx \nonumber\\& = \frac1{4\pi^2}\sum_{j = 1}^n \int_{\mathbb{R}^2} \left(a^2 \partial_1^2 f^{|B_j|}(x) + 2ab \partial_{12} f^{|B_j|}(x) + b^2 \partial_2^2 f^{|B_j|}\right) f^{m -|B_j|}(x) dx \nonumber\\& \quad + \sum_{1 \leq j, j' \leq n}^{j \neq j'} \int_{\mathbb{R}^2} \Big(a^2 \partial_1 f^{|B_j|}(x) \partial_1 f^{|B_{j'}|}(x) + 2ab \partial_{1} f^{|B_j|}(x) \partial_{2} f^{|B_{j'}|}(x) \nonumber \\& \qquad \qquad\qquad \qquad \qquad  \qquad\qquad \qquad \qquad + b^2 \partial_2 f^{|B_j|} \partial_2 f^{|B_{j'}|}\Big)f^{m -|B_j| - |B_j'|}(x) dx.
		\end{align}
		Note that the above expressions for $C_n$ and $D_n$ are null when $n = 1$ since, in that case, $|B_j| = m$ and we can integrate the partial derivatives of $f^{|B_j|}$. Moreover, the sum over $j \neq j'$ is empty. Accordingly, by denoting $C_1$ (resp. $D_1$) the quantities of equation~\eqref{eq_Cn} (resp.~\eqref{eq_Dn}) for $n = 1$, we can add the term $n = 1$ in the sum,
		\begin{align}\label{eq_lim_kappa_m}
			\lim_{r \to \infty} \kappa_m(T_r) = \frac12 \sum_{n =1}^m (-1)^{n-1} (n-1)! \sum_{B_1, \dots, B_n} (C_n + D_n).
		\end{align}
		We pursue the computation~\eqref{eq_Cn} of $C_n$ for $n \geq 1$. Using the fact that $\sum_{j = 1}^{n} |B_j| = m$, we get,
		\begin{align*}
			&2C_n = \frac{1}{4\pi^2} \sum_{j  = 1}^n |B_j| \int_{\mathbb{R}^2} \left(c \partial_1^2 f(x) + d\partial_2^2 f(x) + 2e \partial_{12}^2 f(x)\right)f^{m-1}(x) dx \\&  + \frac{1}{4\pi^2} \sum_{j  = 1}^n |B_j|(|B_j|-1) \int_{\mathbb{R}^2} \left(c (\partial_1 f)^2(x) + d(\partial_2 f)^2(x) + 2e \partial_{1} f(x) \partial_{2} f(x)\right)f^{m-2}(x) dx \\& = q_1 m + q_2 \sum_{j = 1}^n |B_i|^2, 
		\end{align*}
		for some constants $q_1, q_2$ that depend on $f$ and $m$ but not on $n$. Similarly, one can prove using~\eqref{eq_Dn} that there exist similar constants $q_3, q_4$ and $q_5$ such that for $n \geq 1$, 
		$$D_n = q_3 m + q_4\sum_{j = 1}^{n} |B_j|^2 + q_5\sum_{1 \leq j, j' \leq 1}^{j \neq j'} |B_j| |B_j'|.$$
		Finally, recalling~\eqref{eq_lim_kappa_m} and using the combinatorics equations~\eqref{eq_cmb_1},~\eqref{eq_cmb_2},~\eqref{eq_cmb_3} of Lemma~\ref{lemma_cmb}, we obtain $\lim_{r \to \infty}\kappa_m(T_r^0) =0$ for all $m \geq 3$, and the asymptotic normality follows.
	\end{proof}
	
	In the previous theorem, we assume that the perturbations have a finite second-order moment in order to be in a setting where the variance of $T_r^0$ is asymptotically bounded. The next proposition shows that these two properties are in fact equivalent.
	
	\medskip \begin{proposition}\label{prop_var_bounded}
		Let $d =2$ and $f \in \mathcal{S}(\mathbb{R}^2)$. We assume that there exists $x_0 \in \mathbb{R}^2$ such that $\mathcal{F}[f](x_0) \neq 0$. We consider i.i.d. random variables $(\xi_x)_{x \in \mathbb{Z}^2}$ that are a.s. finite. Then, $\operatorname{Var}[T_r^0]$ is bounded if and only if $\mathbb{E}[|\xi_0|^2] < \infty$.
	\end{proposition}
	\begin{proof}
		We have already shown with equation~\eqref{eq_lim_k2} that if $\mathbb{E}[|\xi_0|^2] < \infty$ then $\operatorname{Var}[T_r^0]$ is bounded. It remains to prove the converse. Therefore, assume that $\operatorname{Var}[T_r^0]$ is bounded by $C < \infty$. According to equation~\eqref{eq_var}, this translates to
		\begin{equation*}
			\forall r \geq r_0,~r^2 \int_{\mathbb{R}^2} |\mathcal{F}[f](x)|^2 (1 - |\varphi(x/r)|^2)dx \leq 2 C.
		\end{equation*}
		Let $\xi$ and $\xi'$ be two independent random variables with the same distribution as $\xi_0$. Using equation~\eqref{eq_phi_couplage}, we have,
		$$\int_{\mathbb{R}^2} |\mathcal{F}[f](x)|^2\mathbb{E}\left[r^2\sin^2\left(\frac{(\xi - \xi') \cdot x}{2r}\right)\right] dx \leq C.$$ 
		Then, the Fatou's lemma gives, 
		$$\int_{\mathbb{R}^2} |\mathcal{F}[f](x)|^2\mathbb{E}\left[\left|(\xi - \xi') \cdot x\right|^2\right] dx \leq 4 C.$$
		Since $\xi'$ is a.s. finite there exists $M > 0$ such that $\mathbb{P}[|\xi'| \leq M] > 0$. Subsequently,
		$$\int_{\mathbb{R}^2} |\mathcal{F}[f](x)|^2\mathbb{E}\left[\mathbf{1}_{|\xi'| \leq M}\left|(\xi - \xi') \cdot x\right|^2\right] dx \leq 4 C.$$
		Finally, using the triangle inequality, the independence between $\xi$ and $\xi'$, and the Cauchy-Schwarz inequality, we get 
		\begin{align*}
			\int_{\mathbb{R}^2} |\mathcal{F}[f](x)|^2  \mathbb{E}\left[|\xi.x|^2\right] dx &= \frac{1}{\mathbb{P}[|\xi'| \leq M]}\int_{\mathbb{R}^2} |\mathcal{F}[f](x)|^2  \mathbb{E}\left[\mathbf{1}_{|\xi'| \leq M}|\xi.x|^2\right] dx \\& \leq \frac{16 C}{\mathbb{P}[|\xi'| \leq M]} + 4 M^2 \int_{\mathbb{R}^2} |\mathcal{F}[f](x)|^2 |x|^2 dx < \infty.
		\end{align*}
		By continuity of $\mathcal{F}[f]$, there exists $\rho > 0$ such that $|\mathcal{F}[f]| > 0$ on $B(x_0, \rho)$. Consequently, for all $x \in B(x_0, \rho)$, we have $\mathbb{E}\left[|\xi.x|^2\right]  < \infty.$ Let $e_1 = (0, 1)$ and $e_2 = (1,0)$. Then, 
		$$\mathbb{E}\left[|\xi.e_i|^2\right] = \rho^{-2}\mathbb{E}\left[|\xi.(e_i\rho+ x_0 - x_0)|^2\right] \leq \frac{4}{\rho^2} \left(\mathbb{E}\left[|\xi.(e_i\rho+ x_0)|^2\right]+ \mathbb{E}\left[|\xi.x_0|^2\right]\right) < \infty,$$
		for $i \in \{1, 2\}$, which concludes the proof.
	\end{proof}
	
	\subsubsection{The case of heavy tailed perturbations}\label{sec:clt_ht}
	
	Proposition \ref{prop_var_bounded} implies that if $\mathbb{E}[|\xi_0|^{\nu}] = \infty$ for some $\nu \in (0, 2]$, then the variance of $T_r^0$ is unbounded. Theorem~\ref{thm_tcl_var_unbounded_2_ss} considers the case where $\nu < 2$ and proves that there is always the convergence of a sub-sequence of $T_r^0$ toward a Gaussian random variable, without any further assumptions. Its proof is based on the next lemma, that lower bounds the variance of $T_r^0$ with the tail of the perturbations. 
	
	\medskip \begin{lemma}\label{lemma_var_wtr_tail}
		Let $d =2$ and $f \in \mathcal{S}(\mathbb{R}^2)$. We assume that there exists $x_0 \in \mathbb{R}^2$ such that $f(x_0) \neq 0$. We consider i.i.d. random variables $(\xi_x)_{x \in \mathbb{Z}^2}$. Then, there exist $C > 0$, $\eta > 0$ and $r_0 > 0$ such that for all $r > r_0$:
		$$\operatorname{Var}[T_r^0] \geq C r^2 \mathbb{P}[|\xi_0| \geq \eta r] + O(r^{-1}).$$ 
	\end{lemma}
	\begin{proof}
		According to equation~\eqref{eq_var}, 
		$$\operatorname{Var}[T_r^0] = r^2 \int_{\mathbb{R}^2} |\mathcal{F}[f](x)|^2 (1 - |\varphi(x/r)|^2)dx + O(r^{-1}).$$
		Let $\xi$ and $\xi'$ be two independent random variables with the same distribution as $\xi_0$. Then, using the Fubini's theorem, we have
		\begin{align*}
			\operatorname{Var}[T_r^0]&= r^2 \int_{\mathbb{R}^2} |\mathcal{F}[f](x)|^2 (1 - \mathbb{E}[e^{2\bm{i}\pi(\xi - \xi') \cdot x/r}])dx +O(r^{-1}) \\& = r^2 \mathbb{E}\left[\int_{\mathbb{R}^2} |\mathcal{F}[f](x)|^2 (1 - e^{2\bm{i}\pi(\xi - \xi')\cdot x/r})dx\right] + O(r^{-1}). 
		\end{align*}
		Let $F = \mathcal{F}[|\mathcal{F}[f]|^2]$. The previous equation can be rewritten:
		\begin{align*}
			\operatorname{Var}[T_r^0]&= r^2 \mathbb{E}\left[F(0) - F\left(\frac{\xi - \xi'}r\right)\right] + O(r^{-1}). 
		\end{align*}
		Note that $F$ is real and that $\|F\|_{\infty} = F(0)= \|f\|_2 > 0$. Indeed, 
		$$\overline{F(x)} = \int_{\mathbb{R}^2} |\mathcal{F}[f](k)|^2 e^{-2\bm{i} \pi x.k} dk = \int_{\mathbb{R}^2} |\mathcal{F}[f](-k)|^2 e^{2\bm{i} \pi x.k} dk = F(x).$$
		According to the Riemann-Lebesgue lemma, there exists $\eta > 0$ such that for all $|x| \geq \eta/2$, we have $F(0) - F(x) \geq F(0)/2$. Therefore,
		\begin{align*}
			\operatorname{Var}[T_r^0]& \geq \frac{F(0)}2 r^2 \mathbb{P}[|\xi - \xi'| \geq \eta r/2]+O(r^{-1})\\&  \geq \frac{F(0)}2 r^2 \mathbb{P}[|\xi| \geq \eta r] \mathbb{P}[|\xi'| < \eta r/2] +O(r^{-1}).
		\end{align*} 
		Since $\xi'$ is a.s. finite, there exists $r_0 > 0$ such that for all $r \geq r_0$, $\mathbb{P}[|\xi'| < \eta r/2] \geq 1/2$. This concludes the proof:
		$
			\forall r \geq r_0,~\operatorname{Var}[T_r^0]\geq F(0) r^2 \mathbb{P}[|\xi| \geq \eta r]/4 +O(r^{-1}).
		$
	\end{proof}
	
	\medskip \begin{theorem}\label{thm_tcl_var_unbounded_2_ss}
		We assume that $d = 2$ and consider i.i.d. random variables $(\xi_x)_{x \in \mathbb{Z}^2}$. We suppose that $\mathbb{E}[|\xi_0|^{\nu}]= \infty$ for some $\nu \in (0, 2)$. Let $f \in \mathcal{S}(\mathbb{R}^2)$. Then, there exists a positive sequence $(r_n)_{n \geq 1}$ with $\lim r_n = \infty$, such that
		$$\operatorname{Var}[T_{r_n}^0]^{-1/2} \left(T_{r_n}^0 - r_n^2 \int_{\mathbb{R}^2} f(x) dx \right)\xrightarrow[n \to \infty]{d} \mathcal{N}\left(0, 1\right).$$
	\end{theorem}
	\begin{proof}
		We first show by contradiction that there exists a diverging positive sequence $(r_n)_{n \geq 1}$ such that $\operatorname{Var}[T_{r_n}] \geq C r_n^{2 - \nu} \log(r_n)^{-2}$ and $r_n > 1$ for all $n \geq 1$. Assume that for all $C > 0$ there exists $r_0 > 0$ such that for all $r \geq r_0$, $\operatorname{Var}[T_{r_n}] \leq C r^{2 - \nu} \log(r)^{-2}$. According to Lemma~\ref{lemma_var_wtr_tail}, there exist $\eta > 0$ and $C' < \infty$ such that for all $r \geq r_0$,
		$$r^2 \mathbb{P}[|\xi_0| \geq \eta r] \leq C' r^{2 - \nu}\log(r)^{-2} + O(r^{-1}) \leq (C'+1) r^{2 - \nu} \log(r)^{-2}.$$
		Hence, $\mathbb{P}[|\xi_0| \geq \eta r] \leq C' r^{- \nu} \log(r)^{-2}$ and by the Fubini's theorem, we have
		\begin{align*}
			\mathbb{E}[|\xi_0|^{\nu}] & = \nu \int_0^{\infty} \mathbb{P}[|\xi_0| \geq t] t^{\nu -1} dt \leq \nu \eta^{\nu} \left(\int_0^{r_0} r^{\nu-1} dr + C' \int_{r_0}^{\infty} \frac1{r \log(\eta r)^2} dr\right) < \infty,
		\end{align*}
		which contradicts the fact that $\mathbb{E}[|\xi_0|^{\nu}] = \infty$. Accordingly, there exists an unbounded increasing sequence $(r_n)_{n \geq 1}$ such that $\operatorname{Var}[T_{r_n}] \geq C r_n^{2 - \nu} \log(r_n)^{-2}$. Using the bound~\eqref{eq_bound_kappa_m} on the cumulants, we thus obtain, 
		$$\forall m \geq 3,~\left|\kappa_m(\operatorname{Var}[T_{r_n}]^{-1/2} T_{r_n}^0)\right| \leq C_m \frac{r_n^2 \log(r_n)^{2m}}{r_n^{m(2-\nu)/2}},$$
		for some constant $C_m < \infty$. Subsequently, for $m > 4/(2- \nu)$, $\kappa_m(\operatorname{Var}[T_{r_n}]^{-1/2} T_{r_n}^0)$ converges to $0$ at $n \to \infty$. Finally, according to Corollary~\ref{cor_expec_var}, $$\lim_{r \to \infty }\kappa_1(\operatorname{Var}[T_{r_n}^{-1/2}] (T_{r_n}^0 - r_n^2 \int_{\mathbb{R}^2} f(x)dx)) = 0,$$
		and the second cumulant of the previous random variable is $1$. This concludes the proof.
	\end{proof}
	
	\subsubsection{The case of perturbations having a regularly varying characteristic function}\label{sec:clt_rv_phi}
	
	Without any further assumptions on the perturbations, it is unclear whether Theorem~\ref{thm_tcl_var_unbounded_2_ss} actually occurs for the full sequence $T_r^0$. However, assuming a specific behavior near zero for the characteristic function $\varphi$ of $\xi_0$, we can prove that a central limit theorem holds. This assumption is related to the tail of perturbations (see Lemma~\ref{lemma_char_and_moment}) and allows to obtain a more precise control of the variance of $T_r^0$ than the lower bound provided by Lemma~\ref{lemma_var_wtr_tail}. In addition, it can be satisfied by random variables with infinite moments of order $2$ but finite moments of order strictly less than 2, a case that was excluded in the two previous theorems.
	
	\medskip \begin{theorem}\label{thm_tcl_var_unbounded_2}
		Let $d = 2$ and $f \in \mathcal{S}(\mathbb{R}^2)$. We consider i.i.d. random variables $(\xi_x)_{x \in \mathbb{Z}^d}$. 
		We suppose that the characteristic function $\varphi$ of $(\xi_x)_{x \in \mathbb{Z}^d}$ satisfies $$1 - \varphi(x) \sim L(|x|)|x|^{\alpha},$$
		as $|x| \to 0$, where $\alpha \in (0, 2]$ and $L$ is slowly varying with a limit at $0$.
		Then:
		$$L(1/r)^{-1/2} r^{\frac{\alpha-2}2}\left(T_r^0 - r^2 \int_{\mathbb{R}^2} f(x) dx \right)\xrightarrow[r \to \infty]{d} \mathcal{N}\left(0, \int_{\mathbb{R}^2} |\mathcal{F}[f](x)|^2|x|^{\alpha}dx\right).$$
	\end{theorem}
	\begin{proof}
		We denote $Y_r := L(1/r)^{-1/2} r^{\frac{\alpha-2}2}\left(T_r^0 - r^2 \int_{\mathbb{R}^2} f(x) dx \right)$.
		We first consider the case when $\alpha < 2$. Note that $1 - |\varphi|^2 = (1 - \varphi)\overline{\varphi} + 1 - \overline{\varphi}$. Therefore, the assumption of Proposition~\ref{prop_var} is satisfied and there exist $C, r_0 > 0$ such that 
		$$\forall r \geq r_0,~\operatorname{Var}[T_r^0] \geq C L(1/r) r^{2 - \alpha}.$$
		Since $L$ is slowly varying, for all $\varepsilon > 0$, there exists $r_{\varepsilon} > 0$ such that for $r > r_{\varepsilon}$, $L(1/r) \geq r^{-\varepsilon}$. We choose $\varepsilon < 2 - \alpha$ in order to guarantee that $r^2/(r^{m(2 - \alpha)/2} L(1/r)^{m/2})$ converges to $0$ for $m$ large enough. Accordingly, arguing exactly as in the proof of Theorem~\ref{thm_tcl_var_unbounded_2_ss}, we obtain the asymptotic normality of $Y_r$. Now, when $\alpha = 2$ and $L$ is also bounded, according to Proposition~\ref{prop_var} the variance of $T_r^0$ is bounded. Subsequently, Proposition~\ref{prop_var_bounded} ensures that $\xi_0$ admits a moment of order $2$, and Theorem~\ref{thm_tcl_var_bouded} gives the result. 
		
		The main content of this theorem is thus the setting $\alpha = 2$ and $L$ not bounded. The result will follow from the bound
		\begin{equation}\label{eq_bound_kappa_m_L}
			\forall m \geq 3, ~|\kappa_m(T_r^0)| \leq C_m L(1/r) + o(1)
		\end{equation}
		where $C_m < \infty.$ To obtain it, we use equation~\eqref{eq_kappa_m_Tr_non_asymp}, i.e. for $m \geq 3$, 
 		\begin{equation}\label{eq_kappa_m_L}
			\kappa_m(T_r^0) = r^2\sum_{n =2}^m (-1)^{n-1} (n-1)! \sum_{B_1, \dots, B_n}  \sum_{q = 1}^n  \sum_{S_q} \circledast_{i = 1}^n \{\overline{\mathcal{F}[f^{|B_i|}]}(\varphi_r-1)^{\mu_i}\}(0) + o(1), 
		\end{equation}
		where $\sum_{S_q}$ denotes the sum over the sets $S_q$ of $\{1, \dots, n\}$ of cardinality $|S_q| = q$, $\mu_i = 1$ if $i \in S_q$ and $0$ otherwise with the notations $(\varphi_r -1)^{0} = 1$ and $\varphi_r(x) = \varphi(x/r)$. Let $i^*$ be one index $i$ such that $\mu_{i^*} = 1$. Then, using the fact that $|\varphi_r -1| \leq 2$, we obtain, with the Young's inequality for the convolution, 
		\begin{align*}
			|\circledast_{i = 1}^n \{\overline{\mathcal{F}[f^{|B_i|}]}(\varphi_r-1)^{\mu_i}\}(0)| &\leq \|\mathcal{F}[f^{|B_{i^*}|}](\varphi_r-1)\|_{1} \|\circledast_{i \neq i^*} \{\mathcal{F}[f^{|B_i|}](\varphi_r-1)^{\mu_i}\}\|_{\infty} \\& \leq \|\overline{\mathcal{F}[f^{|B_{i^*}|}]}(\varphi_r-1)\|_{1} 2^{n-1} \|\circledast_{i \neq i^*} |\overline{\mathcal{F}[f^{|B_i|}]}|\|_{\infty}.
		\end{align*}
		Since $1 - \varphi$ is bounded, and satisfies $1 - \varphi(x) \sim L(|x|)|x|^{\alpha}$ as $|x| \to 0$, there exists a constant $C < \infty$ such that for all $x \in \mathbb{R}^2$, then $|1 - \varphi(x)| \leq C L(|x|) |x|^2$. Moreover, according to Proposition~\ref{prop_slwy} for the slowly varying function $L$, there exists $0 < b, r_0, q, C' < \infty$ such that for $|x| \leq b r$ and $r \geq r_0$, then $L(|x|/r) \leq C' \max(|x|^q, |x|^{-1/2}) L(1/r)$. Subsequently,
		\begin{align*}
			\|\mathcal{F}[f^{|B_{i^*}|}]&(\varphi_r-1)\|_{1} \\&\leq \int_{B(0, br)} |\mathcal{F}[f^{|B_{i^*}|}](x)| C L(\frac{|x|}r) |\frac{x}r|^2 dx + 2\int_{B(0, br)^c} |\mathcal{F}[f^{|B_{i^*}|}](x)|dx\\& \leq  C C' \frac{L(1/r)}{r^2} \int_{B(0, br)} |\mathcal{F}[f^{|B_{i^*}|}](x)| \max(|x|^{q+2}, {|x|^{5/2}})dx + o(r^{-2}),
		\end{align*}
		where the $o(r^{-2})$ is a consequence of the fact that $f \in \mathcal{S}(\mathbb{R}^d)$. Summing all the previous inequalities over $B_i$, $q$ and $s \in S_q$, and recalling~\eqref{eq_kappa_m_L}, we finally get~\eqref{eq_bound_kappa_m_L}. Accordingly, for $m \geq 3$, there exists $C_m < \infty$ such that for $r \geq r_0$, 
		$$\left|\kappa_m(L(1/r)^{-1/2}T_r^0)\right| \leq C_m \frac{L(1/r)}{L(1/r)^{m/2}} \to 0,$$
		as $r \to \infty$ since $L$ is unbounded. Since the first two cumulants of $Y_r$ converge respectively to $0$ and $1$, we obtain the asymptotic normality of $Y_r$.
	\end{proof}
	
	\medskip \begin{remark}\label{rmq_non_centered_after_thm_d2}
		In the  previous theorem, the assumption on the characteristic function prohibits perturbations having a non zero mean. However, the results of Section~\ref{sec_stat} allow to extend Theorem~\ref{thm_tcl_var_unbounded_2} to this setting, without requiring additional computations. More details are provided in Remark~\ref{rmk_non_centered_rv}. 
	\end{remark}
	
	\medskip \begin{remark}
		The proof of the previous theorem is unchanged when considering the dimension $d =1$, but with the assumption $\alpha < 1$ or $\alpha = 1$ along with $L$ unbounded. In this second setting, the assumption that $L$ is not bounded is crucial. Indeed, the next section provides examples of non Gaussian limits when $L$ converges to a finite positive limit at zero.
	\end{remark}
	
	\section{Non-Gaussian limits}\label{sec_non_gauss}
	
	\subsection{Lack of a general central limit theorem when the variance is bounded in dimension~1}\label{sec_alpha_1}
	
	In this section, we prove that no general central limit theorem holds in dimension $1$ by considering the case of bounded variance of $T_r^0$. The proof is still based on cumulants, but, contrary to the dimension $2$, no general cancellations occur at the limit. 
	
	\medskip \begin{theorem}\label{thm_no_clt_alpha_1}
		We assume that $d = 1$, $f \in~\mathcal{S}(\mathbb{R})$ and consider i.i.d. random variables $(\xi_x)_{x \in \mathbb{Z}}$ with characteristic function $\varphi$ satisfying $1 - \varphi(x) \sim c |x|$ as $|x| \to 0$, where $c > 0$. 
		\begin{enumerate}
			\setlength\itemsep{0.5em}
			\item \label{i_1_thm_al_1} Then, $T_r^0 - r \int_{\mathbb{R}} f(x) dx$ converges in distribution toward a random variable, which is characterized by its moments. 
			\item \label{i_2_thm_al_1} If $f(x) = e^{-x^2}$, then the limit distribution is not Gaussian.
		\end{enumerate}
	\end{theorem}
	\begin{proof}
	 	The convergence of the two first cumulants of $T_r^0 - r \int_{\mathbb{R}} f(x) dx$ is given by Corollary~\ref{cor_expec_var} and Proposition~\ref{prop_var}. Concerning the cumulants of order $m \geq 3$, one can follow exactly the proof of equation~\eqref{eq_kappa_m_Tr_non_asymp} to obtain for $m \geq 3$, 
		\begin{equation*}
			\kappa_m(T_r^0) = r\sum_{n =2}^m (-1)^{n-1} (n-1)! \sum_{B_1, \dots, B_n}  \sum_{q = 1}^n  \sum_{S_q} \circledast_{i = 1}^n \{\overline{\mathcal{F}[f^{|B_i|}]}(\varphi_r-1)^{\mu_i}\}(0) +o(1), 
		\end{equation*}
		where $\sum_{S_q}$ denotes the sum over the sets $S_q$ of $\{1, \dots, n\}$ of cardinality $|S_q| = q$, $\mu_i = 1$ if $i \in S_q$ and $0$ otherwise with the notations $(\varphi_r -1)^{0} = 1$ and $\varphi_r(x) = \varphi(x/r)$. Using the fact that $|1 - \varphi(x)|$ is bounded by $2$ and is equivalent to $|x|$ near $0$, we obtain that there exists $C_{\varphi} < \infty$ such that for all $x \in \mathbb{R}$, then $|1 - \varphi(x)| \leq C_{\varphi} |x|$. In the following we use the convention $x_0 = x_n = 0$. Accordingly, when $q \geq 2$, 
		\begin{align}\label{eq_d1_qgtr2}
			&r\sum_{S_q} \circledast_{i = 1}^n \left|\{\overline{\mathcal{F}[f^{|B_i|}]}(\varphi-1)^{\mu_i}\}(0)\right| \nonumber\\& \leq \sum_{S_q} \int_{\mathbb{R}^{n-1}} r\prod_{s \in S_q} \left|\varphi\left(\frac{x_s - x_{s-1}}r\right)-1\right| \prod_{i = 1}^n \left|\mathcal{F}[f^{|B_i|}](x_i - x_{i-1})\right|dx \nonumber\\& \leq \frac1r \sum_{S_q} \int_{\mathbb{R}^{n-1}} \prod_{s \in S_q}  C_{\varphi} |x_s - x_{s-1}| \prod_{i = 1}^n \left|\mathcal{F}[f^{|B_i|}](x_i - x_{i-1})\right|dx,
		\end{align}
		which converges to $0$ as $r \to \infty$. When $q = 1$, by the Lebesgue's dominated convergence theorem, we get, as $r \to \infty$,
		\begin{align}\label{eq_d1_q1}
			r \sum_{S_1 \subset \{1, \dots, n\}}& \circledast_{i = 1}^n \{\overline{\mathcal{F}[f^{|B_i|}]}(\varphi_r-1)^{\mu_i}\}(0) \nonumber\\&= \sum_{j = 1}^n \int_{\mathbb{R}^{n-1}} r \left(\varphi\left(\frac{x_j-x_{j-1}}r\right)-1\right)\prod_{i = 1}^n \overline{\mathcal{F}[f^{|B_i|}]}(x_i - x_{i-1}) dx\nonumber \\& \to c \sum_{j = 1}^n \int_{\mathbb{R}^{n-1}}  |x_j - x_{j-1}|\prod_{i = 1}^n \overline{\mathcal{F}[f^{|B_i|}]}(x_i - x_{i-1}) dx.
		\end{align}
		Subsequenlty, using~\eqref{eq_d1_qgtr2} and~\eqref{eq_d1_q1}, we get, for $m \geq 3$, 
		\begin{align*}
			\kappa_m :=\lim_{r \to \infty} \kappa_m(T_r^0) = c \sum_{n = 2}^m a_n \sum_{B_1, \dots, B_n}  \sum_{i = 1}^n \int_{\mathbb{R}^{n-1}} |x_i - x_{i-1}| \prod_{j =1}^n \overline{\mathcal{F}[f^{|B_i|}]}(x_j - x_{j-1}) dx,
		\end{align*}
		where $a_n = (-1)^{n-1} (n-1)!$. To get the point~\ref{i_1_thm_al_1}, we prove that for all $m \geq 3$, $|\kappa_m| \leq C m! q^m$ with $C, q < \infty$, from which the result will follow from the Carleman's condition (see \cite{durrett2019probability}, Section 3.3.5 or Lemma 3.6 of~\cite{gaultier2016fluctuations}). As a first step we bound, for $i \in \{1, \dots, n\}$,
		\begin{align}\label{eq_Ii}
			I_i := \int_{\mathbb{R}^{n-1}}\mathcal{F}[f^{|B_i|}](x_i - x_{i-1}) |x_i - x_{i-1}| \prod_{j = 1, j \neq i}^n \mathcal{F}[f^{|B_j|}](x_j - x_{j-1}) dx.
		\end{align}
		Using the Young's inequality for the convolution and the  Plancherel theorem, we get
		\begin{align*}
			|I_i| &\leq \|\mathcal{F}[f^{|B_i|}](y) y\|_2 \|\circledast_{j = 1, j \neq i}^n \mathcal{F}[f^{|B_j|}]\|_2  \\&\leq |B_i|\|f^{|B_i|-1}f'\|_2 \|f^{m-|B_i|}\|_2 \\& \leq |B_i|\|f\|_{\infty}^{|B_i|-1} \|f'\|_2 \|f\|_{\infty}^{\max(m - |B_i|-1, 0)} \|f\|_2   \\&\leq m (1+\|f\|_{\infty})^{m-2} \|f'\|_2 \|f\|_2.
		\end{align*}
		As a consequence,
		\begin{equation}\label{eq_kappam_1}
			|\kappa_m| \leq \frac{c}2 m (1+\|f\|_{\infty})^{m-2} \|f'\|_2 \|f\|_2 \sum_{n = 1}^m n! \sum_{B_1, \dots, B_n} 1.
		\end{equation}
		To upper bound the sum over $n$, we use the following inequality on the Stirling numbers of the second kind: $\sum_{B_1, \dots, B_n} 1 \leq \binom{m}{n} n^{m-n}$~\cite{rennie1969stirling}.
		Consequently:
		\begin{align}\label{eq_kappam_2}
			\sum_{n = 1}^m n! \sum_{B_1, \dots, B_n} 1 &\leq m! \sum_{n = 1}^m \frac{n^{m-n}}{(m-n)!} \leq m!  \sum_{n = 1}^m \frac{n^{m-n}}{(m-n)^{m-n}}e^{m-n-1} \nonumber\\& =  m!  \sum_{n = 1}^m \frac{(m-n)^n}{n^n}e^{n-1} \leq m! \sum_{n = 1}^m \left(1+\frac{m}{n}\right)^n e^{n-1} \nonumber\\& \leq m! e^{m}\sum_{n = 1}^m  e^{n-1}  \leq (e- 1)^{-1}m! e^{2m -1}.
		\end{align}
		Accordingly, inequalities~\eqref{eq_kappam_1} and~\eqref{eq_kappam_2} imply that for all $m \geq 3$, $|\kappa_m| \leq C m! q^m$ with $C, q < \infty$ that depend on $f$ and $c$ but not on $m$. As a consequence, $T_r^0$ converges toward a random variable, which is characterized by its moments.
		
		Finally, to prove the point~\ref{i_2_thm_al_1}, we show that $\kappa_4 \neq 0$ for $f(x) = e^{-x^2}$. Using the properties of the characteristic function of Gaussian random variables for the computations, we obtain $
			I_i = \sqrt{|B_i|\sum_{j \neq i}|B_j|}/(4\pi)$.
		As a consequence,
		$$\frac{8\pi}{c} \kappa_4 =  \sum_{n = 2}^4 (-1)^{n-1} (n-1)! \sum_{B_1, \dots, B_n} \sum_{i = 1}^n \sqrt{|B_i|(4 - |B_i|)} =  8 \sqrt{3} - 13 > 1/2,$$
		and the limit of $T_r^0$ is not Gaussian when $f(x)= e^{-x^2}$ because.  
	\end{proof}
	
	\medskip \begin{remark}
		The previous proof highlights that the lack of asymptotic normality is due to the irregular behavior of the characteristic function of the perturbation near zero, which is assumed to scale as $1 - c |x|$ as $|x| \to 0$ and to the fact that the asymptotic variance of $T_r^0$ is bounded. This disrupts the combinatorial simplifications, such as the one of Lemma~\ref{lemma_cmb} employed in the proof of Theorem~\ref{thm_tcl_var_bouded}, that arise when considering point processes having smooth second-order properties~\cite{rider2006complex, krishnapur2024stationary}.
	\end{remark}
	
	\subsection{Convergence toward $\alpha$-stable distributions when $d = 1$}\label{sec_no_clt}
	
	Surprisingly, in dimension $1$, and when the variance of $T_r^0$ converges to $0$ as $r \to \infty$, then $T_r^0$ can be asymptotically non Gaussian or Gaussian, depending on the behavior of the characteristic function $\varphi$ of the perturbations near zero. Specifically, if $1 - \varphi(x) \sim c |x|^{\alpha}$ as $|x| \to 0$, where $c > 0$ and $\alpha \in (1, 2]$, then the limit distribution of $T_r^0$ is $\alpha$-stable, and the rate of convergence becomes $r^{(\alpha-1)/\alpha}$. Note that for $\alpha = 2$, we recover a Gaussian limit and the normalization by the square root of the variance of $T_r^0$. 
	
	\medskip \begin{theorem}\label{thm_cv_alpha_stable}
		We assume that $d = 1$ and consider i.i.d. random variables $(\xi_x)_{x \in \mathbb{Z}}$ with characteristic function $\varphi$ satisfying $1 - \varphi(x) \sim c |x|^{\alpha}$ as $|x| \to 0$, where $c > 0$ and $\alpha \in (1, 2]$. Let $f \in \mathcal{S}(\mathbb{R})$. Then:
		$$r^{\frac{\alpha - 1}{\alpha}}\left(T_r^0 - r \int_{\mathbb{R}}f(x) dx\right) \xrightarrow[r \to \infty]{d} \left(c\int_{\mathbb{R}} |f'(x)|^{\alpha} dx\right)^{1/\alpha} S_{\alpha},$$
		where $S_{\alpha}$ is an $\alpha$-stable law, with characteristic function $\mathbb{E}[e^{2\bm{i}\pi k S_{\alpha}}] = e^{-|k|^{\alpha}}$.
	\end{theorem}
	
	\medskip \begin{remark}
		As already mentioned in Remark~\ref{rmq_non_centered_after_thm_d2}, the above theorem can be extended to the case of non centered perturbations. More details are provided in Remark~\ref{rmk_non_centered_rv}.
	\end{remark}
	\medskip
	Before proving this result, we introduce a few heuristics that will be made rigorous in the remaining of this section. The key remark is that for $\alpha >~1$, the perturbations are enough localized to decouple $x$ and $\xi_x$, using the Taylor expansion
	\begin{equation}\label{eq_heur_taylor}
		T_r^0 \simeq \sum_{x \in \mathbb{Z}} f\left(\frac{x}r\right) +  \sum_{x \in \mathbb{Z}}  f'\left(\frac{x}r\right) \frac{\xi_x}r.
	\end{equation}
	Suppose, for simplicity, that the perturbations $(\xi_x)_{x \in \mathbb{Z}}$ follow an $\alpha$-stable $S_{\alpha}$ distribution. According to the stability of the sums of independent stable distributions, the term $\sum_{x \in \mathbb{Z}}  f'(x/r) \xi_x/r$ should be close to $r^{-1} \left(\sum_{x \in \mathbb{Z}} |f'(x/r)|^{\alpha}\right)^{1/\alpha} S_{\alpha}$. Then, recognizing a Riemann sum, we anticipate:
	\begin{equation}\label{eq_heur_riemann}
		\lim_{r\to \infty}\frac1r\sum_{x \in \mathbb{Z}} |f'(x/r)|^{\alpha} = \int_{\mathbb{R}} |f'(x)|^{\alpha} dx.
	\end{equation}
	The previous discussion suggests 
	$
		T_r^0 \simeq  r \int_{\mathbb{R}} f(x) dx + r^{-(\alpha - 1)/\alpha} \left(\int_{\mathbb{R}} |f'(x)|^{\alpha} dx\right)^{1/\alpha} S_{\alpha},
	$
	which is Theorem~\ref{thm_cv_alpha_stable}. In the remaining of this section, we prove that the previous arguments can be made rigorous in dimension $1$ and with $\alpha \in (1, 2]$. 
	
	The first step is to split $T_r^0$ between a main term (i.e. $|x| \leq r^{\gamma}$) and a remainder term (i.e. $|x| > r^{\gamma}$), where $\gamma >1$:
	$$r^{\frac{\alpha -1}{\alpha}} T_r^0 = r^{\frac{\alpha -1}{\alpha}} \sum_{|x| \leq r^{\gamma}} f\left(\frac{x + \xi_x}r\right) + r^{\frac{\alpha -1}{\alpha}} \sum_{|x| > r^{\gamma}} f\left(\frac{x + \xi_x}r\right).$$
	
	Lemma~\ref{lemma_al_big_main_term} studies the main term for which the approximation~\eqref{eq_heur_taylor} can be quantified. It introduces a parameter $p > 1$ that has to be small enough (regarding $\gamma$ and $\alpha$) to ensure that the quantity $\Delta$ of the next lemma vanishes when $r \to \infty$. However, the possible lack of moments of the random variables $(\xi_x)_{x \in \mathbb{Z}}$ forces $1+1/p < \alpha$. Lemma~\ref{lemma_al_sup_1remainder_term} proves that for $\gamma$ big enough the remainder term $r^{\frac{\alpha -1}{\alpha}} \sum_{|x| > r^{\gamma}} f\left((x + \xi_x)/r\right)$ converges in probability to $0$ when $r \to \infty$. It introduces an other parameter $\beta$ that should be big enough to ensure that the remainder vanishes as $r \to \infty$ (but, also with the restriction $\beta < \alpha$). As the choice of the parameters $\gamma$, $p$ and $\beta$ is not direct, we postpone it to Lemma~\ref{lemma_numero}. Finally, Lemma~\ref{lemma_riemann_f_alpha} proves~\eqref{eq_heur_riemann}.
	
	\medskip \begin{lemma}\label{lemma_al_big_main_term}
		We assume that $\alpha > 1$. Let $\gamma > 1$, $t \in \mathbb{R}$, and $p >1$ satisfying $1+1/p < \alpha$. We define $\Delta$ by:
		$$\prod_{|x| \leq r^{\gamma}} \mathbb{E}\left[\exp\left(\bm{i} t f\left(\frac{x+\xi_x}r\right) r^{\frac{\alpha-1}\alpha}\right)\right] - \prod_{|x| \leq r^{\gamma}} \mathbb{E}\left[\exp\left(\bm{i} t  \left(f\left(\frac{x}r\right) + f'\left(\frac{x}r\right) \frac{\xi_x}r\right)r^{\frac{\alpha-1}\alpha}\right)\right].$$
		Then, there exists a constant $C < \infty$ such that
		\begin{equation}\label{eq_al_sup_1_main_term}
			\left|\Delta\right| \leq C \frac{r^{\gamma + (\alpha-1)/\alpha}}{r^{1 +1/p}}.
		\end{equation}
	\end{lemma}
	\begin{proof}
		We use the bound, valid for all $(a_i,b_i)_{1\leq i \leq n} \in \mathbb{C}^{2n}$ satisfying $|a_i| \leq 1$ and $|b_i| \leq1$,
		\begin{equation}\label{eq_diff_prod}
			\left|\prod_{i = 1}^n a_i - \prod_{i = 1}^n b_i\right| \leq \sum_{i = 1}^n |a_i - b_i|,
		\end{equation}
		to obtain
		$$\left|\Delta\right|  \leq r^{\frac{\alpha - 1}{\alpha}}  |t| \sum_{|x| \leq r^{\gamma}} \mathbb{E}\left|f\left(\frac{x+\xi}r\right) - f\left(\frac{x}r\right) + f'\left(\frac{x}r\right) \frac{\xi}r\right|.$$
		To handle the lack of moments of $\xi$, we use a Taylor expansion with integral remainder term and Holder's inequality:
		\begin{align*}
			\left|f\left(\frac{x+\xi}r\right) - f\left(\frac{x}r\right) + f'\left(\frac{x}r\right) \frac{\xi}r\right| &\leq \frac12 \left|\int_{\frac{x}r}^{\frac{x+\xi}r} \left|\frac{x}r - \tau\right| |f''(\tau)| d\tau\right| \\& \leq \frac12 \left|\int_{\frac{x}r}^{\frac{x+\xi}r} \left|\frac{x}r - \tau\right|^{p} \right|^{1/p} \left(\int_{\mathbb{R}} |f''(\tau)|^{p/(p-1)} d\tau\right)^{1 - 1/p} \\& = \frac1{2(p+1)^{1/p}} \left(\int_{\mathbb{R}} |f''(\tau)|^{p/(p-1)} d\tau\right)^{1 - 1/p} \frac{|\xi|^{1+1/p}}{r^{1+1/p}}.
		\end{align*}
		This gives~\eqref{eq_al_sup_1_main_term} with a finite constant since $1+1/p < \alpha$ (see Lemma~\ref{lemma_char_and_moment}):
		$$C = \frac{|t| \mathbb{E}[|\xi|^{1+1/p}]}{(p+1)^{1/p}} \left(\int_{\mathbb{R}} |f''(\tau)|^{p/(p-1)} d\tau\right)^{1 - 1/p} < \infty.$$
	\end{proof}
	
	\medskip \begin{lemma}\label{lemma_al_sup_1remainder_term}
		We assume that $\alpha > 1$. Let $\gamma > 1$, $t \in \mathbb{R}$, and $\beta$ such that, $1 < \beta < \alpha$. Then, there exists a constant $C < \infty$ such that for $r^{\gamma} > 2$:
		\begin{equation}\label{eq_al_sup_1remainder_term}
			r^{\frac{\alpha -1}{\alpha}} \mathbb{E}\left[\left|\sum_{|x| > r^{\gamma}} f\left(\frac{x + \xi_x}r\right) \right|\right] \leq C \frac{r^{{(\alpha -1)/\alpha}}}{r^{\gamma(\beta-1)}}.
		\end{equation}
	\end{lemma}
	\begin{proof}
		Let $k > 1$ that will be chosen in the end of the proof. Using the fact that $f \in \mathcal{S}(\mathbb{R})$ and the Markov inequality, we obtain for $|x|> 2$:
		\begin{align*}
			\mathbb{E}\left[ |f|\left(\frac{x+\xi}r\right)\right]&= \mathbb{E}\left[\mathbf{1}_{|x+\xi| \geq |x|/2}|f|\left(\frac{x+\xi}r\right)\right] +   \mathbb{E}\left[\mathbf{1}_{|x+\xi| < |x|/2}|f|\left(\frac{x+\xi}r\right)\right]\\& \leq \sup_{y \in \mathbb{R}} |y^{k} f(y)| 2^{k} \frac{r^{k}}{|x|^{k}} + \|f\|_{\infty} \mathbb{P}\left[|\xi| \geq |x|/2\right] \\& \leq \sup_{y \in \mathbb{R}} |y^{k} f(y)| 2^{k} \frac{r^{k}}{|x|^{k}} + \|f\|_{\infty} 2^{\beta} \frac{\mathbb{E}[|\xi|^{\beta}]}{|x|^{\beta}}.
		\end{align*}
		As $\beta < \alpha$ Lemma~\ref{lemma_char_and_moment} ensures that $\mathbb{E}[|\xi|^{\beta}] < \infty$. Hence, using $\beta > 1$ and $k > 1$, we obtain that there exists some constant $C < \infty$ such that for all $r^{\gamma} >2$:
		$$r^{\frac{\alpha -1}{\alpha}} \mathbb{E}\left[\left|\sum_{|x| > r^{\gamma}} f\left(\frac{x + \xi_x}r\right)\right|\right] \leq C r^{\frac{\alpha -1}{\alpha}}\left(\frac{r^k}{r^{\gamma(k-1)}} + \frac{1}{r^{\gamma(\beta -1)}}\right),$$
		Eventually, by choosing $k = \gamma \beta/(\gamma -1)> \gamma/(\gamma -1) > 1$, we obtain~\eqref{eq_al_sup_1remainder_term}.
	\end{proof}
	
	The next lemma confirms that it is possible to choose the parameters $\gamma, p$ and $\beta$, introduced in the two previous lemmas, such that both $\Delta$ and $r^{\frac{\alpha -1}{\alpha}} \sum_{|x| > r^{\gamma}} f\left(\frac{x + \xi_x}r\right)$ vanish when $r \to \infty$. 
	\medskip \begin{lemma}\label{lemma_numero}
		We assume that $\alpha > 1$. There exists $\gamma > 1,~p > 1$ and $\beta > 1$ such that:
		\begin{enumerate}
			\item \label{c1} $1+1/p < \alpha$,
			\item \label{c3} $\gamma + \frac{\alpha -1}{\alpha} < 1+1/p,$
			\item \label{c2}$1 < \beta < \alpha$,
			\item \label{c4} $\frac{\alpha -1}{\alpha} < \gamma(\beta -1)$.
		\end{enumerate}
	\end{lemma}
	\begin{proof}
		As $\alpha > 1$, there exists $\varepsilon > 0$ such that $\alpha^{-1} < \varepsilon < 1$. We choose $p = (\varepsilon(\alpha -1))^{-1}$, which satisfies $p > 1$ and the constraint~\ref{c1}:
		$$p = (\varepsilon(\alpha -1))^{-1} \geq (\varepsilon(2-1))^{-1} = \varepsilon^{-1} > 1, \qquad
		1+1/p = 1+\varepsilon(\alpha -1) < 1+ \alpha -1 = \alpha.$$
		We now choose $\gamma$ such that $1 < \gamma < \alpha^{-1} + \varepsilon(\alpha -1)$. It is possible since:
		$$\alpha^{-1} + \varepsilon(\alpha -1) > 1 \iff \varepsilon(\alpha - 1) > (\alpha-1)/\alpha \iff \varepsilon > \alpha^{-1}.$$
		This choice satisfies the constraint~\ref{c3}:
		$$\gamma + \frac{\alpha -1}{\alpha} <  \alpha^{-1} + \varepsilon(\alpha -1) + \frac{\alpha -1}{\alpha} = 1 + 1/p.$$
		We finally choose $\beta$ such that $1 + \frac{\alpha-1}{\alpha \gamma} < \beta < \alpha$. We can do so since $1 + \frac{\alpha-1}{\alpha \gamma} < 1 + \alpha -1 = \alpha$. Constraints~\ref{c2} and~\ref{c4} are satisfied with this choice, concluding the proof.
	\end{proof}
	
	The next lemma states that $\int_{\mathbb{R}} |f'(x)|^{\alpha} dx$ can be approximated by a convenient Riemann sum.
	\medskip \begin{lemma}\label{lemma_riemann_f_alpha}
		Let $\alpha > 1$ and $f \in \mathcal{S}(\mathbb{R})$. Then:
		$$\lim_{r \to \infty} r^{-1} \sum_{x \in \mathbb{Z}} |f'(x/r)|^{\alpha} = \int_{\mathbb{R}} |f'(x)|^{\alpha} dx.$$
	\end{lemma}
	\begin{proof}
		Let $1 < q < \alpha$. Using the fact that $f \in \mathcal{S}(\mathbb{R})$, we truncate the sum:
		\begin{align*}
			r^{-1} \sum_{|x| > r^q} |f'(x/r)|^{\alpha} &\leq r^{-1} \sup_{y \in \mathbb{R}} |y^{2} |f'(y)|^{2\alpha}| \sum_{|x| > r^q}  \frac{r^{2}}{|x|^{2}} \leq \sup_{y \in \mathbb{R}} |y^{2} |f'(y)|^{2 \alpha }| \frac1{r^{q-1}} \to 0,
		\end{align*}
		when $r \to \infty$. Accordingly it suffices to prove that
		$\lim_{r \to \infty} r^{-1} \sum_{|x| \leq r^q} |f'(x/r)|^{\alpha} = \int_{\mathbb{R}} |f'(x)|^{\alpha} dx.$
		To do so, we use the standard Riemann-sum procedure:
		\begin{align*}
			A &:= r^{-1} \sum_{|x| \leq r^q} |f'(x/r)|^{\alpha} - \int_{\mathbb{R}} |f'(x)|^{\alpha} dx \\& = \sum_{|x| \leq r^q} \int_{x/r}^{\frac{x+1}r} \left(|f'(x/r)|^{\alpha} - |f'(y)|^{\alpha}\right) dy - \int_{\mathbb{R}\setminus[-r^{q-1}, (r^q+1)/r]}|f'(y)|^{\alpha} dy.
		\end{align*}
		Since $|f'|^{\alpha} \in L^1(\mathbb{R})$ and $q > 1$, the integral over $\mathbb{R}\setminus[-r^{q-1}, (r^q+1)/r]$ vanishes when $r \to \infty$. Moreover, $1 < \alpha \leq 2$, so we have, by concavity of $x \mapsto |x|^{\alpha -1}$ and by the mean value~theorem:
		\begin{align*}
			\left||f'(\frac{x}r)|^{\alpha} - |f'(y)|^{\alpha}\right| &\leq |f'(\frac{x}r)|^{\alpha-1}\left||f'(\frac{x}r)| - |f'(y)|\right|+|f'(y)|\left||f'(\frac{x}r)|^{\alpha-1} - |f'(y)|^{\alpha-1}\right| \\& \leq \|f'\|_{\infty}^{\alpha-1} \|f''\|_{\infty}\left|\frac{x}r - y\right| + \|f'\|_{\infty} \|f''\|_{\infty}^{\alpha -1} \left|\frac{x}r - y\right|^{\alpha -1} \\& \leq C \left(\left|\frac{x}r - y\right| +  \left|\frac{x}r - y\right|^{\alpha -1}\right),
		\end{align*}
		where $C  = \|f'\|_{\infty}^{\alpha-1} \|f''\|_{\infty} +\|f'\|_{\infty} \|f''\|_{\infty}^{\alpha -1}$. Accordingly, we obtain when $r \to \infty$:
		\begin{align*}
			|A| &\leq o(1) +  C \sum_{|x| \leq r^q} \int_{x/r}^{\frac{x+1}r} \left(\left|\frac{x}r - y\right| +  \left|\frac{x}r - y\right|^{\alpha -1}\right)dy \\& \leq o(1) +  C \sum_{|x| \leq r^q} \left(\frac12 r^{-2} + \frac1{\alpha} r^{-\alpha}\right) \leq o(1) + \frac{4C}{\alpha} \frac1{r^{\alpha-q}} \to 0.
		\end{align*}
		This concludes the proof.
	\end{proof}
	
	Gathering previous lemmas, we now prove Theorem~\ref{thm_cv_alpha_stable}. 
	\begin{proof}[Proof of Theorem~\ref{thm_cv_alpha_stable}]
		We consider the parameters $\beta, p, \gamma$ of Lemma~\ref{lemma_numero}. According to Lemma~\ref{lemma_al_sup_1remainder_term}, the remainder term $r^{\frac{\alpha - 1}{\alpha}} \sum_{|x| > r^{\gamma}} f\left(\frac{x + \xi_x}r\right)$ converges to 0 in $L^1(\mathbb{P})$ and thus in probability. Consequently, according to the Slutsky's lemma, it is enough to prove that:
		$$r^{\frac{\alpha - 1}{\alpha}} \left(\sum_{|x| \leq r^{\gamma}} f\left(\frac{x + \xi_x}r\right) - r \int_{\mathbb{R}} f(x)dx\right) \xrightarrow[r \to \infty]{d} \left(\int_{\mathbb{R}} |f'(x)|^{\alpha} dx\right)^{1/\alpha} S_{\alpha}.$$
		According to~\eqref{eq_al_sup_1_main_term} and with $p$ and $\beta$ of Lemma~\ref{lemma_numero}, we only need to prove that, for all $t \in \mathbb{R}$
		\begin{align*}
			\phi_r(t) := \exp\left(- 2\pi\bm{i}t  r^{\frac{\alpha-1}\alpha} r \int_{\mathbb{R}} f(x) dx\right)
			\times \prod_{|x| \leq r^{\gamma}} \mathbb{E}\left[\exp\left(2\pi \bm{i} t  \left(f\left(\frac{x}r\right) + f'\left(\frac{x}r\right) \frac{\xi_x}r\right)r^{\frac{\alpha-1}\alpha}\right)\right]
		\end{align*}
		converges toward $\exp\left(-\int_{\mathbb{R}} |f'(x)|^{\alpha} dx |t|^{\alpha}\right)$.
		Using the independence of $(\xi_i)_{i \in \mathbb{Z}^d}$, we get
		\begin{align*}
			\phi_r(t) &= \exp\left( 2\pi \bm{i} t  r^{(\alpha-1)/\alpha} \left(\sum_{|x| \leq r^{\gamma}} f(x/r) - r \int_{\mathbb{R}} f(x) dx\right)\right) \prod_{|x| \leq r^{\gamma}} \varphi\left(t \frac{f'(x/r)}{r^{1/\alpha}}\right).
		\end{align*}
		We first prove that:
		\begin{equation}\label{eq_sup_1_moyenne}
			\lim_{r \to \infty} r^{(\alpha-1)/\alpha} \left(\sum_{|x| \leq r^{\gamma}} f(x/r) - r \int_{\mathbb{R}} f(x) dx\right) = 0.
		\end{equation}
		Applying the Poisson summation formula (see Theorem~\ref{thm_poisson}) to $f \in \mathcal{S}(\mathbb{R})$, we get
		\begin{align*}
			\sum_{|x| \leq r^{\gamma}} f\left(\frac{x}r\right) &= \sum_{x \in \mathbb{Z}} f\left(\frac{x}r\right) - \sum_{|x| > r^{\gamma}} f\left(\frac{x}r\right) = r \sum_{x \in \mathbb{Z}} \mathcal{F}[f](rx) - \sum_{|x| > r^{\gamma}} f\left(\frac{x}r\right).
		\end{align*}
		Consequently:
		$$ r^{\frac{\alpha -1}{\alpha}} \left(\sum_{|x| \leq r^{\gamma}} f\left(\frac{x}r\right) - r \int_{\mathbb{R}} f(x) dx\right)= r^{\frac{\alpha -1}{\alpha}} \left(r \sum_{x \in \mathbb{Z}\setminus\{0\}} \mathcal{F}[f](rx) - \sum_{|x| > r^{\gamma}} f\left(\frac{x}r\right)\right).$$
		Lemma~\ref{lemma_tail} applied to $\xi \equiv 0$ implies $r^{\frac{\alpha -1}{\alpha}} r \sum_{x \in \mathbb{Z}\setminus\{0\}} \mathcal{F}[f](rx) \to 0$, as $r \to \infty$. Moreover 
		\begin{align}\label{eq_remain_f_r_gamma}
			\forall k > 1,~r^{\frac{\alpha -1}{\alpha}}  \sum_{|x| > r^{\gamma}} |f|(x/r) \leq r^{\frac{\alpha -1}{\alpha}} \frac{r^k}{r^{\gamma(k-1)}} \sup_{y \in \mathbb{R}}|y^k f(y)| 2 (k-1)^{-1}.
		\end{align}
		Choosing $k > \left(\frac{\alpha -1}{\alpha} + 1 + \gamma\right)(\gamma-1)^{-1}$, we obtain~\eqref{eq_sup_1_moyenne}. The last part consists in proving that
		\begin{equation}\label{eq_limit_phi_r}
			\lim_{r \to \infty} \prod_{|x| \leq r^{\gamma}} \varphi\left(t \frac{f'(x/r)}{r^{1/\alpha}}\right) = \exp\left(- |t|^{\alpha}\int_{\mathbb{R}} |f'(x)|^{\alpha} dx\right).
		\end{equation}
		Using the bound~\eqref{eq_diff_prod} we obtain, for all $r \geq 2c |t|^{\alpha} \|f'\|_{\infty} =: r_0$:
		\begin{multline*}
			\Big|\prod_{|x| \leq r^{\gamma}} \varphi\left(t \frac{f'(x/r)}{r^{1/\alpha}}\right) - \prod_{|x| \leq r^{\gamma}} \left(1 - c |t|^{\alpha} \frac{|f'(x/r)|^{\alpha}}{r} \right)\Big| \\ \leq \sum_{|x| \leq r^{\gamma}} \left|\varphi\left(t \frac{f'(x/r)}{r^{1/\alpha}}\right) - \left(1 - c |t|^{\alpha} \frac{|f'(x/r)|^{\alpha}}{r}\right)\right|. 
		\end{multline*}
		Let $\varepsilon > 0$. The assumption on $1 - \varphi$ near the origin ensures that there exists $\delta > 0$ such that for all $|k| \leq \delta$, then $\left|\varphi(k) - (1 - c |k|^{\alpha})\right| \leq \varepsilon |k|^{\alpha}.$ Thus, for $r \geq \left(|t| \|f'\|_{\infty}/\delta\right)^{\alpha} \wedge r_0 =: r_1$,
		\begin{align*}
			\left|\prod_{|x| \leq r^{\gamma}} \varphi\left(t \frac{f'(x/r)}{r^{1/\alpha}}\right) - \prod_{|x| \leq r^{\gamma}} \left(1 - c |t|^{\alpha} \frac{|f'(x/r)|^{\alpha}}{r} \right)\right| \leq \varepsilon  |t|^{\alpha} \sum_{|x| \leq r^{\gamma}} \frac{|f'(x/r)|^{\alpha}}{r}. 
		\end{align*}
		Using Lemma~\ref{lemma_riemann_f_alpha} thereafter, we obtain that, there exists $r_2 \geq r_1$ such that for all $r \geq r_2$:
		\begin{align}\label{eq_phi_r_1}
			\left|\prod_{|x| \leq r^{\gamma}} \varphi\left(t \frac{f'(x/r)}{r^{1/\alpha}}\right) - \prod_{|x| \leq r^{\gamma}} \left(1 - c |t|^{\alpha} \frac{|f'(x/r)|^{\alpha}}{r} \right)\right| \leq 2\varepsilon  |t|^{\alpha}  \int_{\mathbb{R}} |f'(x)|^{\alpha} dx.
		\end{align}
		Since $\log(1+x) \sim x$ as $|x| \to 0$, we get with the same arguments leading to~\eqref{eq_phi_r_1} but applied to $\log(1+x)$ instead of $\varphi$, that there exists $r_3 \geq r_2$ such that for all $r \geq r_3$:
		\begin{align}\label{eq_phi_r_2}
			\left|\sum_{|x| \leq r^{\gamma}} \log\left(1 - c |t|^{\alpha} \frac{|f'(x/r)|^{\alpha}}{r}\right) - r^{-1}\sum_{|x| \leq r^{\gamma}}c |t|^{\alpha} |f'(x/r)|^{\alpha} \right| \leq 2\varepsilon  |t|^{\alpha}  \int_{\mathbb{R}} |f'(x)|^{\alpha} dx.
		\end{align}
		Accordingly, Lemma~\ref{lemma_riemann_f_alpha} and~\eqref{eq_phi_r_2} yield
		$$\lim_{r \to \infty} \prod_{|x| \leq r^{\gamma}} \left(1 - c |t|^{\alpha} \frac{|f'(x/r)|^{\alpha}}{r} \right) = \exp\left(- |t|^{\alpha}c \int_{\mathbb{R}} |f'(x)|^{\alpha} dx\right).$$
		Finally, using~\eqref{eq_phi_r_1}, we obtain the convergence~\eqref{eq_limit_phi_r}, which concludes the proof.
		\end{proof}	
	
	\section{Generalization to stationary perturbed lattices}\label{sec_stat}
	
	In this section, we prove that the statements of the theorems of the previous sections are unchanged when considering the stationary perturbed lattice $\Phi^{1}$, defined in~\eqref{eq:def_phi}.
	
	\medskip \begin{theorem}\label{thm_cv_stat}
		Theorems~\ref{thm_tcl_var_unbouded},~\ref{thm_tcl_var_bouded},~\ref{thm_tcl_var_unbounded_2_ss},~\ref{thm_tcl_var_unbounded_2},~\ref{thm_no_clt_alpha_1},~\ref{thm_cv_alpha_stable} and~\ref{thm_cv_stat} remain valid if $T_r^0$ is replaced by $T_r^1$.
	\end{theorem}
	\medskip
	Theorem~\ref{thm_cv_stat} is still based on the Poisson summation formula and uses several times the fact that, the uniform random variable $U$, in the Fourier domain, will introduce terms $e^{2\bm{i} \pi U.x}$ with $x \in \mathbb{Z}^d$. This leads to useful cancellations. To prove Theorem~\ref{thm_cv_stat}, we use two lemmas. Lemma~\ref{lemma_exp_exp_stat} derives the first moment of $T_r^1$ and Lemma~\ref{lemma_diff_non_stat_stat} proves that the difference $T_r^0 - T_r^1$ converges to $0$ in $L^{2}(\mathbb{P})$ fast enough regarding $\operatorname{Var}[T_r^0]$.
	
	\medskip \begin{lemma}\label{lemma_exp_exp_stat}
		Let $d \geq 1$, $f \in \mathcal{S}(\mathbb{R}^d)$,  $(\xi_x)_{x \in \mathbb{Z}^d}$ i.i.d. and $U$ be a uniform random variable on $[-1/2, 1/2]^d$ which is independent of $(\xi_x)_{x \in \mathbb{Z}^d}$. Then, 
		\begin{equation}\label{eq_expec_stat}
			\mathbb{E}[T_r^{1}] = r^d\int_{\mathbb{R}^d} f(x) dx.
		\end{equation}
	\end{lemma}
	\begin{proof}
		Equation~\eqref{eq_expec_stat} is a direct application of the Campbell's averaging formula for stationary point processes (see Theorem 1.2.5 of~\cite{baccelli2020random}). However, we propose a different approach, that will be useful in the following. Using~\eqref{eq_kappa_int} with $m= 1$ and $f(\cdot + U)$, we get (where $f_r = f(\cdot/r)$):
		\begin{align}\label{eq_exp_T1}
			\mathbb{E}[T_r^{1}|U] = \sum_{x \in \mathbb{Z}^d} \overline{\mathcal{F}[f_r(\cdot + U/r)]}(x) \varphi(x)  = r^d \sum_{x \in \mathbb{Z}^d} e^{2\bm{i} \pi U.x} \overline{\mathcal{F}[f]}(rx) \varphi(x).
		\end{align}
		Using the Fubini's theorem and $\mathbb{E}[e^{2\bm{i} \pi U.x}] = \mathbf{1}\{x = 0\}$, we obtain
		\begin{equation*}
			\mathbb{E}[T_r^{1}]= r^d \sum_{x \in \mathbb{Z}^d} \mathbb{E}[e^{2\bm{i} \pi U.x}] \overline{\mathcal{F}[f]}(rx) \varphi(x) = r^d \mathcal{F}[f](0) \varphi(0) = r^d \int_{\mathbb{R}^d} f(x) dx.
		\end{equation*} 
	\end{proof}
	
	\medskip \begin{lemma}\label{lemma_diff_non_stat_stat}
		Let $d \geq 1$ and $f \in \mathcal{S}(\mathbb{R}^d)$. We consider $(\xi_x)_{x \in \mathbb{Z}^d}$ i.i.d. random variables. Let $U$ be a uniform random variable on $[-1/2, 1/2]^d$ which is independent of $(\xi_x)_{x \in \mathbb{Z}^d}$. Then, we have when $r \to \infty$,
		\begin{equation}\label{eq_diff1}
			\mathbb{E}[(T_r^0 - T_r^{1})^2] = o\left(\operatorname{Var}[T_r^0]\right).
		\end{equation}
	\end{lemma}
	\begin{proof}
		Proposition~\ref{cor_expec_var} and Lemma~\ref{lemma_exp_exp_stat} give, for all $\beta > 0$
		\begin{align*}
			\mathbb{E}[(T_r^0 - T_r^{1})^2] &= \left(\mathbb{E}[T_r^0 - T_r^{1}]\right)^2 + \operatorname{Var}[T_r^0 - T_r^{1}] =  O(r^{-\beta}) + \operatorname{Var}[T_r^0 - T_r^{1}].
		\end{align*}
		To bound the variance term, we use the law of the total variance. Using equation~\eqref{eq_kappa_int} with $m =1$ and equation~\eqref{eq_exp_T1} we~get
		\begin{align*}
			\mathbb{E}[T_r^0 - T_r^{1}|U] &= r^d \sum_{x \in \mathbb{Z}^d} \overline{\mathcal{F}[f]}(rx)\varphi(x) - r^d \sum_{x \in \mathbb{Z}^d} e^{2\bm{i}\pi U.x} \overline{\mathcal{F}[f]}(rx)\varphi(x) \\& = r^d \sum_{x \in \mathbb{Z}^d \setminus\{0\}} (1-e^{-2\bm{i}\pi U.x}) \overline{\mathcal{F}[f](rx)}\varphi(x).
		\end{align*}
		According to Lemma~\ref{lemma_tail_0}, $\mathbb{E}[T_r^0 - T_r^{1}|U] \leq C' r^{-\beta}$ where $C'<\infty$ is non random, since $|1-e^{-2\bm{i}\pi U.x}| \leq 2$. Hence $\operatorname{Var}[\mathbb{E}[T_r^0 - T_r^{1}|U]] \leq (C')^2 r^{-2\beta}$. Concerning $\mathbb{E}[\operatorname{Var}[T_r^0 - T_r^{1}|U]]$, we use equation~\eqref{eq_var} with $g_r := f(\cdot + U/r) - f(y)$ to obtain:
		\begin{align*}
			\operatorname{Var}[T_r^0 - T_r^{1}|U] &= r^d \int_{\mathbb{R}^d} |\mathcal{F}[g_r](k)|^2 (1 - |\varphi(k/r)|^2) dk + O(r^{-\beta}) \\&= r^d \int_{\mathbb{R}^d} |e^{-2\bm{i} \pi U.k/r} -1|^2|\mathcal{F}[f](k)|^2 (1 - |\varphi(k/r)|^2) dk+O(r^{-\beta}),
		\end{align*}
		where the term $O(r^{-\beta})$ is bounded by a non random constant. Then, using the fact that $|e^{-2\bm{i} \pi U.k/r} -1|^2 = 4 \sin^2(\pi k.U/r)$ and the Cauchy-Schwarz inequality, we get
		\begin{equation*}\label{eq_diff3}
			\mathbb{E}[(T_r^0 - T_r^{1})^2] \leq \pi^2 r^{d-2} \int_{\mathbb{R}^d}\mathcal{F}[f](k)|^2 |k|^2 (1 - |\varphi(k/r)|^2) dk + O(r^{-\beta}).
		\end{equation*}
		We define
		$$h(r) := \left(\int_{\mathbb{R}^d}|\mathcal{F}[f](k)|^2 |k|^2 (1 - |\varphi(k/r)|^2) dk\right)^{1/2}.$$ 
		The Lebesgue's dominated convergence theorem ensures that $\lim_{r \to \infty} h(r) = 0$. So we can apply Corollary~\ref{lemma_var_d3} with this function $h$, yielding, for some $C > 0$,
		\begin{equation}\label{eq_o_diff}
			\frac{\mathbb{E}[(T_r^0 - T_r^{1})^2]}{\operatorname{Var}[T_r^0]} \leq C \left(h(r) + C_{\beta} \frac{r^{2-\beta-d}}{h(r)}\right).
		\end{equation} 
		Finally, we remark that,
		$h(r)^2 = r^{-d} \operatorname{Var}\left[\sum_{x \in \mathbb{Z}^d} g\left((x + \xi_x)/r\right)\right],$
		with the Schwartz function $g = (4\pi)^{-2} \sum_{i = 1}^d \partial_{x_i} f$. Consequently, we can apply Corollary~\ref{lemma_var_d3} with the function $r \mapsto r^{-1}$ to get: $h(r)^2 \geq C r^{-d} r^{d-2} r^{-1} = r^{-3}$, for some constant $C > 0$. Hence, by choosing $\beta= 3$, inequality~\eqref{eq_o_diff} proves~\eqref{eq_diff1}.
	\end{proof}
	
	We are now in position to prove Theorem~\ref{thm_cv_stat}.
	\begin{proof}[Proof of Theorem~\ref{thm_cv_stat}]
		We first handle the case where the normalization is the standard deviation. We assume that $X_r := \operatorname{Var}[T_r^0]^{-1/2}(T_r^0 - r^d \int_{\mathbb{R}^d} f(x) dx)$ converges in distribution to some random variable $X$. We decompose:
		$$\operatorname{Var}[T_r^0]^{-1/2}(T_r^1 - r^d \int_{\mathbb{R}^d} f(x) dx) = \operatorname{Var}[T_r^0]^{-1/2}(T_r^1 - T_r^0) + X_r.$$
		Lemma~\ref{lemma_diff_non_stat_stat} ensures that $\operatorname{Var}[T_r^0]^{-1/2}(T_r^1 - T_r^0)$ converges to $0$ in $L^2(\mathbb{P})$ and thus in probability. Consequently, according to the Slutsky's lemma, $\operatorname{Var}[T_r^0]^{-1/2}(T_r^1 - r^d \int_{\mathbb{R}^d} f(x) dx)$ converges to $X$ in distribution. 
		
		We assume now that $d = 1$ and $1 - \varphi(x) \sim c |x|^{\alpha}$ for $\alpha \in(1, 2]$ and $c > 0$. According to Theorem~\ref{thm_cv_alpha_stable}, $Y_r := r^{(\alpha -1)/\alpha}(T_r^0 - r \int_{\mathbb{R}} f(x) dx)$ converges in distribution to some random variable $Y$. Then, according to Lemma~\ref{lemma_diff_non_stat_stat} and Proposition~~\ref{prop_var}, 
		$$\mathbb{E}\left[r^{2(\alpha-1)/\alpha}(T_r^1 - T_r^0)^2\right]
		\leq r^{2(\alpha-1)/\alpha} \times o(r^{1-\alpha}) = r^{(\alpha-1)(2/\alpha -1)}\times o(1).$$ 
		Since $\alpha \leq 2$, the previous bound and the Slutsky's lemma imply $r^{\frac{\alpha - 1}{\alpha}}(T_r^1 - r \int_{\mathbb{R}} f(x) dx) \to Y$ in distribution, which concludes the proof.
	\end{proof}
	
	\medskip \begin{remark}\label{rmk_non_centered_rv}
		Note that in Lemma~\ref{lemma_diff_non_stat_stat} and Theorem~\ref{thm_cv_stat}, we only used the fact that $U$ is square integrable. Hence, the previous results can be applied to a constant random variable. Accordingly, Theorems~\ref{thm_tcl_var_unbounded_2} and~\ref{thm_cv_alpha_stable} remain valid if the perturbations $(\xi_x)_{x \in \mathbb{Z}^d}$ admit a non zero mean $\mu$. Indeed, it suffices to consider the linear statistic corresponding to $(\xi_x - \mu)_{x \in \mathbb{Z}^d}$ for which both Theorems~\ref{thm_tcl_var_unbounded_2} and~\ref{thm_cv_alpha_stable} apply, and then use Theorem~\ref{thm_cv_stat} with $U \equiv \mu$.
	\end{remark}
	
	\section{Results related to the Poisson summation formula}\label{sec_technical}
	
	In this section, we prove Lemma~\ref{lemma_poisson_k} which is the main tool for computing the cumulants of $T_r^0$, Lemma~\ref{lemma_moment} that ensures that $T_r^0$ possesses moments of all orders and Lemma~\ref{lemma_tail}, that yields asymptotic simplifications in the results of Lemma~\ref{lemma_poisson_k}. Before proving these three results, we need the next preliminary lemma. 
	\medskip \begin{lemma}\label{lemma_fourier_convol}
		Let $n \geq 1$ and $(f_i)_{i = 1 \dots, n}$ be a family of infinitely differentiable functions, with derivatives of all orders that are bounded and integrable. Moreover we assume that $\forall x \in \mathbb{R}^d, |f_i(x)| \leq C_i(1+|x|^{2d})^{-1}$, for all $i = 1, \dots, n$. Let $\xi$ be a random variable with characteristic function $\varphi$. Let $s: x \in \mathbb{R}^d \mapsto \circledast_{i = 1}^n \{\overline{\mathcal{F}[f_i]} \varphi\}(x)$. Then, $s \in L^1(\mathbb{R}^d)$ and:
		$$\forall k \in \mathbb{R}^d,~ \mathcal{F}[s](x) = \prod_{i = 1}^n \mathbb{E}[f_i(x+ \xi)].$$
		Moreover, for all $k \in \mathbb{N}$, $\sup_{x \in \mathbb{R}^d}||y|^k s(y)| < \infty$.  
	\end{lemma}
	\begin{proof}
		According to Young's inequality for the convolution and using the fact that the functions $(f_i)_{1 \leq i \leq n }$ possess integrable derivatives of all orders, we get
		$$\|s\|_1 \leq \prod_{i = 1}^n \|\mathcal{F}[f_i] \varphi\|_1 \leq \prod_{i = 1}^n \|\mathcal{F}[f_i]\|_1 < \infty.$$
		Consequently, $s \in L^1(\mathbb{R}^d)$ and $ \mathcal{F}[s]$ is defined everywhere. Moreover, using the Fubini's theorem, we obtain:
		\begin{align*}
		\overline{\mathcal{F}[f_i]}(y) \varphi(y) =  \overline{\mathcal{F}[f_i]}(y) \mathbb{E}[e^{2\bm{i} \pi \xi \cdot y}] =\mathbb{E}[\mathcal{F}[f_i(\xi - \cdot)]](y)  = \mathcal{F}[\mathbb{E}[f_i(\xi - \cdot)]](y).
		\end{align*}
		By the Lebesgue's dominated convergence theorem,  $\mathbb{E}\left[f_i(\xi - \cdot)\right]$ is infinitely differentiable with integrable derivatives, whence $\mathcal{F}\left[\mathbb{E}\left[f_i(\xi - \cdot)\right]\right] \in L^1(\mathbb{R}^d)$. Hence, we can apply the Fourier inversion theorem, see, e.g.,~\cite{folland2009fourier} and we obtain, for all $k \in \mathbb{R}^d$, 
		$$\mathcal{F}[s](k) = \mathcal{F}\left[ \circledast_{i = 1}^n  \mathcal{F}[\mathbb{E}[f_i(\xi - \cdot)]]\right] = \prod_{i =1}^n \mathcal{F}\left[ \mathcal{F}\left[\mathbb{E}\left[f_i(\xi - \cdot)\right]\right]\right(k) = \prod_{i = 1}^n \mathbb{E}[f_i(k+ \xi)].$$
		As $\|\mathcal{F}[s]\|_1 \leq \prod_{i = 2}^n \|f_i\|_{\infty} \mathbb{E}\left[\|f_1(\cdot + \xi)\|_1\right] = \|f_1\|_1 \prod_{i = 2}^n \|f_i\|_{\infty} < \infty$ (with the product equal to $1$ if $n = 1$) and $s \in L^1(\mathbb{R}^d)$ the Fourier inversion theorem gives:
		$$\forall y \in \mathbb{R}^d \setminus\{0\}, ~s(y) = \int_{\mathbb{R}^d} \prod_{i = 1}^n \mathbb{E}[f_i(x+ \xi)] e^{-2\bm{i} \pi x \cdot y} dx.$$
		Using the assumptions on $f_1, \dots, f_n$ and the Lebesgue's dominated convergence theorem, we obtain by integrations by parts, that for all $k \in \mathbb{N}$, $j  \in \{1, \dots,d\}$ and $y \in \mathbb{R}^d \setminus\{0\}$:
		\begin{align*}
			|y_j^k s(y)| &= \left|(2\bm{i} \pi)^{-k} \int_{\mathbb{R}^d}  \partial_{x_j}^k\left[ \prod_{i = 1}^n \mathbb{E}[f_i(x+ \xi)]\right] e^{-2\bm{i} \pi x \cdot y} dx\right| \\& \leq (2 \pi)^{-k} \sum_{|B_1| + \dots + |B_n| = k} \frac{n!}{|B_1|! \dots |B_n|!} \int_{\mathbb{R}^d}\prod_{i = 1}^n \mathbb{E}\left[\left| \partial_{x_j}^k f_i(x+ \xi)\right|\right] dx \\& \leq (2 \pi)^{-k} \sum_{|B_1| + \dots + |B_n| = k} \frac{n!}{|B_1|! \dots |B_n|!} \| \partial_{x_j}^k f_1\|_1 \prod_{i = 2}^n \| \partial_{x_j}^k f_i\|_{\infty}.
		\end{align*}
		We conclude using $|y|^k \leq d^k\sum_{j = 1}^d |y_j|^{k}$ and summing the previous inequalities.
	\end{proof}
	
	Lemma~\ref{lemma_poisson_k} proves that the Poisson summation formula can be applied in the context of perturbed lattices with no assumptions on the random variable $\xi$, as soon as the functions $f_1, \dots, f_n$ are smooth and enough localized. It uses the following version of the Poisson formula as stated in Theorem 3.2.8 of~\cite{grafakos2008classical}.
	
	\medskip \begin{theorem}\label{thm_poisson}
		Let $g : \mathbb{R}^d \to \mathbb{R}$ and assume that:
		\begin{enumerate}
			\item g is continuous,
			\item there exist $C < \infty$ and $\delta > 0$ such that $|g(x)| \leq C(1+|x|)^{-d - \delta}$,
			\item $\sum_{x \in \mathbb{Z}^d} |\mathcal{F}[g](x)| < \infty$.
		\end{enumerate}
		Then, $$\sum_{x \in \mathbb{Z}^d} g(x) = \sum_{x \in \mathbb{Z}^d} \mathcal{F}[g](x).$$
	\end{theorem}
	Note that, to prove the next lemma, one cannot directly apply the previous theorem to $g(x) := \prod_{i = 1}^n \mathbb{E}[f_i(x+\xi)]$, as $g$ does not necessarily satisfies its assumption 2 due to possible slow decay of the tails of the perturbations. But, one can prove that $\mathcal{F}[g]$ satisfies the conditions of this aforementioned Poisson summation formula.
	
	\medskip \begin{lemma}\label{lemma_poisson_k}
		Let $n \geq 1$ and $(f_i)_{i = 1 \dots, n}$ be a family a infinitely differentiable functions with derivatives of all orders that are bounded and integrable. Moreover, we assume that $\forall x \in \mathbb{R}^d, |f_i(x)| \leq C_i(1+|x|^{2d})^{-1}$, for all $i = 1, \dots, n$ and where $C_i < \infty$. Let $\xi$ be a random variable with characteristic function $\varphi$. Then, the following equality holds, where both terms are finite:
		$$\sum_{x \in \mathbb{Z}^d} \prod_{i = 1}^n \mathbb{E}[f_i(x+\xi)] =  \sum_{x \in \mathbb{Z}^d} \circledast_{i = 1}^n \{\overline{\mathcal{F}[f_i]} \varphi\}(x).$$
	\end{lemma}
	\begin{proof}
		We apply Theorem~\ref{thm_poisson} to the continuous function $s = \circledast_{i = 1}^n \{\overline{\mathcal{F}[f_i]} \varphi\}$. To check condition 3, we first use Lemma~\ref{lemma_fourier_convol}:
		$$\sum_{x \in \mathbb{Z}^d} |\mathcal{F}[s](x)| = \sum_{x \in \mathbb{Z}^d} \left|\prod_{i = 1}^n \mathbb{E}[f_i(x+\xi)]\right| \leq \prod_{i = 2}^{n} \|f_i\|_{\infty} \sum_{x \in \mathbb{Z}^d}  |\mathbb{E}[f_1(x+\xi)]|.$$
		Using the decay assumption on $f_1$ and the Fubini's theorem, we obtain:
		$$\sum_{x \in \mathbb{Z}^d} |\mathbb{E}[f_1(x+\xi)]| \leq C_1 \mathbb{E}\left[\sum_{x \in \mathbb{Z}^d} \frac1{1+|x + \xi|^{2d}} \right].$$
		Now, we use Theorem~\ref{thm_poisson} with the function $g_y: x \mapsto  (1+|x +y|^{2d})^{-1}$ in order to prove that $F:y \mapsto \sum_{x \in \mathbb{Z}^d} (1+|x + y|^{2d})^{-1}$ is bounded. We check that $\sum_{x \in \mathbb{Z}^d} |\mathcal{F}[g_y](x)| < \infty$ since $g_y$ is infinitely  differentiable with integrable derivatives. Accordingly, point 3 is satisfied. The points 1 and 2 are also trivially satisfied by $g_y$. Consequently
		$$|F(y)| = \left|\sum_{x \in \mathbb{Z}^d} e^{-2i\pi y \cdot x}\mathcal{F}\left[(1+|\cdot|^{2d})^{-1}\right](x) \right| \leq \sum_{x \in \mathbb{Z}^d} \left|\mathcal{F}\left[(1+|\cdot|^{2d})^{-1}\right]\right|(x) < \infty.$$	 
		We have thus obtained, 
		$$\sum_{x \in \mathbb{Z}^d} \left|\prod_{i = 1}^n \mathbb{E}[f_i(x+\xi)]\right| \leq C_1 \|F\|_{\infty}  \prod_{i = 2}^{n} \|f_i\|_{\infty} < \infty,$$
		which shows condition 3 of Theorem \ref{thm_poisson} for the function $s$. According to Lemma~\ref{lemma_fourier_convol}, there exists $C < \infty$ such that: $\forall x \in \mathbb{R}^d, ~|s(x)| \leq C(1+|x|^{2d})^{-1}$, so the point 2 of Theorem~\ref{thm_poisson} is satisfied by $s$ and using the fact that $\mathcal{F}[s]=\prod_{i = 1}^n \mathbb{E}[f_i(\cdot+ \xi)]$ (see Lemma~\ref{lemma_fourier_convol}), we get the result. 
	\end{proof}

	Next lemma ensures that for $i = 0, 1$, $T_r^i$ possesses moments of all orders. As a corollary we obtain that for all $k \in \mathbb{N}$, and Borel set $B$ of $\mathbb{R}^d$ with finite Lebesgue measure then $\mathbb{E}[(\sum_{x \in \Phi^i} \mathbf{1}_{B}(x))^k] < \infty$. Indeed, $\mathbf{1}_{B}$ can be bounded above by a Schwartz function.
	\medskip \begin{lemma}\label{lemma_moment}
		Let $k \in \mathbb{N}$, $f$ be a Schwartz function, $(\xi_x)_{x \in \mathbb{Z}^d}$ be i.i.d. random variables with characteristic function $\varphi$ and $U$ be a uniform random variable over $[-1/2, 1/2]^d$ which is independent of $(\xi_x)_{x \in \mathbb{Z}^d}$. Then, $$\mathbb{E}\left[\left(\sum_{x \in \mathbb{Z}^d}|f|\left(x + \xi_x + U\right) \right)^k\right] + \mathbb{E}\left[\left(\sum_{x \in \mathbb{Z}^d}|f|\left(x + \xi_x\right) \right)^k\right] < \infty.$$
	\end{lemma}
	\begin{proof}
		To exploit the independence between the random variables $(\xi_x)_{x \in \mathbb{Z}^d}$, we expand the sum to the power $k$ into sums over distinct indexes:
		\begin{align*}
			\mathbb{E}\left[\left(\sum_{x \in \mathbb{Z}^d}|f|\left(x+\xi_x\right) \right)^k\right] = \sum_{n = 1}^k \sum_{l_1 + \dots + l_n = k} \mathbb{E}\left[\sum_{x_1, \dots, x_n \in \mathbb{Z}^d}^{\neq} \prod_{i = 1}^{n}|f|^{l_i}\left(x_i + \xi_{x_i}\right) \right].
		\end{align*}
		Let $\xi$ be a random variable having characteristic function $\varphi$. Since $f$ is Schwartz, there exists $C < \infty$ such that $|f|(x) \leq C(1 + |x|)^{-2d} := g(x)$ and by independence, we have 
		\begin{align*}
			\mathbb{E}\left[\left(\sum_{x \in \mathbb{Z}^d}|f|\left(x+\xi_x\right) \right)^k\right]  &\leq \sum_{n = 1}^k \sum_{l_1 + \dots + l_n = k} \mathbb{E}\left[\sum_{x_1, \dots, x_n \in \mathbb{Z}^d}^{\neq} \prod_{i = 1}^{n}g^{l_i}\left(x_i + \xi_{x_i}\right) \right] \\& = \sum_{n = 1}^k \sum_{l_1 + \dots + l_n = k} \sum_{x_1, \dots, x_n \in \mathbb{Z}^d}^{\neq} \prod_{i = 1}^{n} \mathbb{E}\left[g^{l_i}\left(x_i + \xi\right)\right].
		\end{align*}
		The functions $(g^{l_i})_{i = 1, \dots, k}$ satisfy the assumptions of Lemma~\ref{lemma_poisson_k}. Accordingly,
		\begin{align*}
			\sum_{x \in \mathbb{Z}^d}  \mathbb{E}\left[g^{l_i}\left(x + \xi\right)\right] &\leq \sum_{x \in \mathbb{Z}^d} |\mathcal{F}[g^{l_i}]|(x) |\varphi|(x) \\& \leq \sup_{y \in \mathbb{R}^d} |(1+|y|)^{2d} \mathcal{F}[g^{l_i}](y)| \sum_{x \in \mathbb{Z}^d}  (1+|x|)^{-2d}< \infty.
		\end{align*}
		In the last line, we used the fact that the functions $(g^{l_i})_{1 \leq i \leq n}$ are infinitely differentiable with integrable derivatives. Finally, by considering $f((\cdot+U)/r)$ instead of $f$ and using the same computations, but with the expectation $\mathbb{E}[\cdot|U]$, and then taking the expectation with respect to $U$, we obtain the result.
	\end{proof}

	Next lemma, coupled with Lemma~\ref{lemma_poisson_k}, allows to derive simple asymptotic expression for the cumulants of $T_r^0$ as $r \to \infty$. 
	\medskip \begin{lemma}\label{lemma_tail}
		Let $r > 0$, $n \geq 1$, $(f_i)_{1\leq i \leq n} \in \mathcal{S}(\mathbb{R}^d)^n$ and $\varphi$ be a characteristic function. We note $f_{i, r}(x) = f_i(x/r)$. Then, for any $\beta > 0$, there exists $C(\beta) < \infty$ such that:
		$$\sum_{x \in \mathbb{Z}^d\setminus\{0\}} \left|\circledast_{i = 1}^n \{\overline{\mathcal{F}[f_{i, r}]}\varphi\}(x)\right| \leq C(\beta) r^{-\beta}.$$
	\end{lemma}
	\begin{proof}
		Let $\beta > 0$. We first consider the case $n = 1$. Using the fact that $f_1$ is Schwartz, we obtain for some $C < \infty$:
		$$\sum_{x \in \mathbb{Z}^d\setminus\{0\}} \left|\overline{\mathcal{F}[f_{1, r}]}(x)\varphi(x)(x)\right| \leq \sum_{x \in \mathbb{Z}^d\setminus\{0\}} r^{d} |\mathcal{F}[f_1] (rx)| \leq r^d \sum_{x \in \mathbb{Z}^d} C |rx|^{-\beta-d} = C_{\beta} r^{-\beta},$$
		where $C_{\beta} = C \sum_{x \in \mathbb{Z}^d} |x|^{-d - \beta} < \infty$.
		For the $n \geq 2$ case, we prove that for $x \neq 0$, then $\left|\circledast_{i = 1}^n \{\mathcal{F}[f_{i, r}]\varphi\}(x)\right| \leq C r^{-\beta} |x|^{-d - \beta}$. Using $|\varphi| \leq 1$ and changes of variables we get, $$\left|\circledast_{i = 1}^n \{\overline{\mathcal{F}[f_{i, r}]}\varphi\}(x)\right| \leq \circledast_{i = 1}^n \left|\mathcal{F}[f_{i, r}]\right|(x) = r^d \circledast_{i = 1}^n \left|\mathcal{F}[f_{i}]\right|(rx).$$
		Using that for $x \neq 0$, $$1 = \frac{|rx|^{d + \beta}}{|rx|^{d +\beta}} \leq 2^{d+\beta}\frac{|rx - x_{n-1}|^{d + \beta} + |x_{n-1}|^{d+\beta}}{|rx|^{d + \beta}},$$
		we get $I \leq 2^{d+\beta} r^{-\beta} |x|^{-d - \beta}(A+B),$ with $A,B$ given respectively by:
		$$\int_{(\mathbb{R}^d)^{n-1}} |rx - x_{n-1}|^{d+\beta}|\mathcal{F}[f_n](r x - x_{n-1}) \prod_{i = 1}^{n-1} \mathcal{F}[f_i](x_i-{x_{i-1}})| dx_1 \dots d_{x_{n-1}},$$
		$$\int_{(\mathbb{R}^d)^{n-1}} |x_{n-1}|^{d+\beta}|\mathcal{F}[f_n](r x - x_{n-1}) \prod_{i = 1}^{n-1} \mathcal{F}[f_i](x_i-{x_{i-1}})| dx_1 \dots d_{x_{n-1}},$$
		where, by convention $x_0 = 0$. To bound $A$ we use the Young's inequality for the convolution:
		$$A \leq \sup_{y \in \mathbb{R}^d} ||y|^{d+\beta} |\mathcal{F}[f_n](y)| \prod_{i = 1}^{n-1} \|\mathcal{F}[f_i]\|_1.$$
		Using the inequality $$|x_{n-1}|^{d+\beta} \leq (n-1)^{d+\beta}\left(|x_{n-1} - x_{n-2}|^{d+\beta} + \dots + |x_2 - x_1|^{d+\beta} + |x_1|^{d+\beta}\right),$$
		we bound $B$ similarly:
		$$B \leq \|\mathcal{F}[f_n]\|_{\infty} (n-1)^{d+\beta} \sum_{i = 1}^{n-1} \sup_{y \in \mathbb{R}^d} |\mathcal{F}[f_i](y) |y|^{d+\beta}| \prod_{j = 1, j \neq i}^{n-1} \|\mathcal{F}[f_j]\|_1.$$
		We have thus proved that for all $x \neq 0$, $\left|\circledast_{i = 1}^n \{\overline{\mathcal{F}[f_{i, r}]}\varphi\}(x)\right| \leq C r^{-\beta} |x|^{-d - \beta}$, for some $C < \infty$. The result follows by summing over $x \in \mathbb{Z}^d\setminus\{0\}$ and using $d + \beta > d$.
	\end{proof}

	\medskip \begin{remark}
		Note that the proof of Lemmas \ref{lemma_fourier_convol}, \ref{lemma_poisson_k} and \ref{lemma_moment} can be adapted to the setting where the functions $(f_i)_{1 \leq i \leq n}$ are not smooth but still localized, by assuming regularity on the perturbation $\xi$ (in order to ensure that its characteristic function $\varphi$ decay fast enough). However, Lemma \ref{lemma_tail} really requires smoothness of the functions $(f_i)_{1 \leq i \leq n}$. Indeed, even for Gaussian perturbations and for non pathological functions, one can construct a counter example to its conclusion when $n = 2$ and for functions belonging to a Sobolev space; we refer to the example constructed in~\cite{yakir2021fluctuations}, already mentioned in Remark~\ref{rmk_smooth_f}. 
	\end{remark}
	
	 \appendix
	\section{Results related to the assumption on the characteristic function.}\label{sec_hyp_phi}
	
	In this appendix only, to alleviate the notations, the characteristic function of a real random variable $X$ is defined by $\varphi(k) = \mathbb{E}[e^{\bm{i} k \cdot X}]$.
	\medskip \begin{proposition}\label{prop_lim_al}
		Let $X$ be a random variable taking values on $\mathbb{R}^d$, with characteristic function $\varphi$ satisfying near the origin: $1 - \varphi(k) \sim L(|k|)|k|^{\alpha}$ for $|k| \to 0$, where $\alpha > 2$ and $L$ is slowly varying. Then, $X \equiv 0$ a.s.
	\end{proposition}
	\begin{proof}
		The proof follows the idea of Lemma~\ref{lemma_var_d3}. Since $X$ is real, we can express $1 - \Re \varphi(k) = \mathbb{E}[1 - \cos(k \cdot X)] = 2\mathbb{E}[\sin^2(k \cdot X/2)]$. Moreover, because $1 - \varphi \in L^{\infty}(\mathbb{R}^d)$ and  $1 - \varphi(k) \sim L(|k|)|k|^{\alpha} $ for $|k| \to 0$, there exists a constant $C < \infty$ such that, $1 - \Re{\varphi(k)} \leq |1 - \varphi(k)| \leq C L(|k|)|k|^{\alpha}$ for all $k \in \mathbb{R}^d \setminus\{0\}$. As a consequence, we have
		$$\forall k \in \mathbb{R}^d,~\mathbb{E}[\sin^2(k \cdot X/2)] \leq \frac{CL(|k|)|k|^{\alpha}}{2}.$$
		Let $k = \nu k_0$ with $\nu > 0$ and $k_0 \in \mathbb{R}^d\setminus\{0\}$. Using the Fatou's lemma and $\alpha > 2$, we get
		$$0 \leq \frac14\mathbb{E}[|X\cdot k_0|^2]=\mathbb{E}[\underset{\nu \to 0}{\operatorname{\lim\inf}} \frac{\sin^2(\nu k_0\cdot X/2)}{|\nu|^2}] \leq \frac{C|k_0|^{\alpha}}{2} \underset{\nu \to 0}{\operatorname{\lim\inf}} \nu^{\alpha -2} L(|k_0| \nu)= 0 .$$ 
		Hence, for all $i \in \{1, \dots, d\}$, a.s. $X\cdot e_i \equiv 0$, where $(e_i)_{1 \leq i \leq d}$ is the canonical basis of $\mathbb{R}^d$ and therefore $X \equiv 0$ a.s.
	\end{proof}
	
	\medskip \begin{lemma}\label{lemma_char_and_moment}
		Let $X$ be a random variable taking values on $\mathbb{R}^d$, with characteristic function $\varphi$ satisfying near the origin: $1 - \varphi(k) \sim L(|k|)|k|^{\alpha}$ for $|k| \to 0$, where $0 < \alpha \leq 2$ and $L$ is slowly varying. Then, for all $\mu < \alpha$, $\mathbb{E}||X|^{\mu}] < \infty.$
	\end{lemma}
	\begin{proof}
		Arguing as in the beginning of the proof of Proposition~\ref{prop_lim_al} and using the fact that $t \in [\pi/2, \pi) \mapsto \sin^2(t/2)$ is increasing and lower bounded by $1/2$ we have
		$$\forall k \in \mathbb{R}^d,~\frac12\mathbb{E}[\mathbf{1}_{\pi/2 \leq |k \cdot X| < \pi}] \leq  \mathbb{E}[\sin^2(k \cdot X/2)] \leq \frac{C|k|^{\alpha} L(|k|)}{2}.$$
		Let $0 < \varepsilon  < \alpha$. Since $L$ is slowly varying, there exist $r_0 > 0$ such that for $|k| \leq r_0$, then $L(|k|) \leq |k|^{- \varepsilon}$. We denote $\alpha' = \alpha - \varepsilon$ and $i_0 \in \mathbb{N}$ such that $\pi 2^{-i_0} \leq r_0$. Consequently, for $k = \pi 2^{-i} e_j$ where $(e_j)_{1 \leq j \leq d}$ is the canonical basis of $\mathbb{R}^d$, we obtain
		$$\forall i \geq i_0+1,~ \mathbb{P}(|X\cdot e_j| \geq 2^{i-1}) -  \mathbb{P}(|X\cdot e_j| \geq 2^{i}) \leq \frac{C \pi^{\alpha'}}{2^{i\alpha'}}.$$
		Summing the previous inequalities and using $\mathbb{P}(|X\cdot e_j| \geq 2^{i}) \to 0$ as $i \to \infty$, we get
		$$\forall i\geq i_0,~\mathbb{P}(|X\cdot e_j| \geq 2^{i}) \leq \frac{C \pi^{\alpha'}}{1 - 2^{-\alpha'}} \frac1{2^{\alpha'(i+1)}}.$$
		Let $t \geq \log_2(i_0)$ and $i \geq i_0$ such that $2^{i} \leq t \leq 2^{i+1}$. Then,
		$$\mathbb{P}(|X\cdot e_j| \geq t) \leq \mathbb{P}(|X \cdot e_j| \geq 2^{i}) \leq \frac{C \pi^{\alpha'}}{1 - 2^{-\alpha'}} \frac1{2^{\alpha'(i+1)}} \leq \frac{C \pi^{\alpha'}}{1 - 2^{-\alpha'}} \frac1{t^{\alpha'}}.$$
		Let $\mu < \alpha$. Choosing $\varepsilon < \alpha - \mu$, we obtain the conclusion with the Fubini's theorem: 
		\begin{align*}
			\mathbb{E}||X|^{\mu}] & \leq \sum_{j = 1}^d \mu \int_{0}^{\infty} \mathbb{P}(|X \cdot e_j| \geq t/d) t^{\mu-1}dt \\& \leq d^2 \log_2(i_0) + d^{1 - \alpha'} \frac{C \pi^{\alpha'}\mu}{1 - 2^{-\alpha'}} \int_{d \log_2(i_0)}^{\infty} \frac{dt}{t^{\alpha - \varepsilon - \mu +1}} < \infty.
		\end{align*}
		
	\end{proof}
	
	\section{Results related to slowly varying functions.}\label{app_slwy}
	
	\begin{proposition}\label{prop_slwy}
		Let $L$ be a slowly varying function. Then, there exist $0 < b, C, q < \infty$ such that for all $r > 0$ and $1 \leq |x| \leq b r$, 
		\begin{equation}\label{eq_slwy_1}
			L(|x|/r) \leq C |x|^q L(1/r). 
		\end{equation}
		Moreover, for all $\eta > 0$, there exists $r_0 > 0$ such that for all $r \geq r_0$, and $|x| \leq 1$, 
		\begin{equation}\label{eq_slwy_2}
			L(|x|/r) \leq C |x|^{-1/2} L(1/r). 
		\end{equation}
	\end{proposition}
	\begin{proof}
		The bounds~\eqref{eq_slwy_1} and~\eqref{eq_slwy_2} appears in the proof of Theorem 3 of~\cite{soshnikov2002gaussian}. For completeness, we detail the arguments in the following. According to the representation theorem of slowly varying functions (see Theorem 1.2. of~\cite{seneta2006regularly}), 
		$$\forall t \in (0, b],~L(t)= \exp\left(c(t) + \int_{b^{-1}}^{t^{-1}} \frac{\varepsilon(\tau)}{\tau} d\tau\right),$$
		where $b > 0$, $c$ is bounded and has a finite limit when $t \to 0$, and $\varepsilon$ is continuous and converges to $0$ as $\tau \to \infty$. Accordingly, for $1 \leq |x| \leq b r$, 
		\begin{align*}
			\frac{L(|x|/r)}{L(1/r)} &\leq \exp\left(2 \|c\|_{\infty} + \int_{r/|x|}^r \|\varepsilon\|_{\infty} \frac{d\tau}{\tau} \right) = e^{2 \|c\|_{\infty}} |x|^{\|\varepsilon\|_{\infty}},
		\end{align*}
		which gives~\eqref{eq_slwy_1}. It remains to prove~\eqref{eq_slwy_2}. We use the representation theorem and the fact that $\varepsilon(\tau)$ converges to $0$ as $\tau \to \infty$ (which ensure that there exists $r_0 \geq 1/b$, such that for all $\tau \geq r_0$, then $|\varepsilon(\tau)| \leq 1/2$), to get for $r \geq r_0$ and $|x| \le 1$
		\begin{align*}
			\frac{L(|x|/r)}{L(1/r)} &\leq \exp\left(2 \|c\|_{\infty} + \frac12\int_r^{r/|x|} \frac{d\tau}{\tau} \right) = e^{2 \|c\|_{\infty}} |x|^{-1/2}.
		\end{align*}
	\end{proof}

	\section*{Acknowledgment}
	
	The author sincerely thanks Frédéric Lavancier for thoroughly reviewing the first version of the manuscript and providing valuable suggestions that improved its clarity and presentation.
	
	\bibliographystyle{acm}

\end{document}